\documentclass{./MyPreprintStyleFile}

\begin{document}

\title{Convolution Products and $R$-Matrices for KLR Algebras of Type $B$}
\author{Ruari Walker}
\address{Universit\'{e} Paris Diderot-Paris VII \\ B\^{a}timent Sophie Germain \\ 75205 Paris Cedex 13 \\ France}
\email{ruari.walker@imj-prg.fr}
\maketitle

\thispagestyle{empty}

\begin{abstract}
In this paper we define and study convolution products for modules over certain families of VV algebras. We go on to study morphisms between these products which yield solutions to the Yang-Baxter equation so that in fact these morphisms are $R$-matrices. We study the properties that these $R$-matrices have with respect to simple modules with the hope that this is a first step towards determining the existence of a (quantum) cluster algebra structure on a natural quotient of $\btg$, the $\mathbb{Q}(q)$-algebra defined by Enomoto and Kashiwara, which the VV algebras categorify. 
\end{abstract}

\section*{Introduction}
The \let\thefootnote\relax\footnote{Supported in part by the European Research Council in the framework of the
The European Union H2020 with the Grant ERC 647353 QAffine.} Khovanov-Lauda-Rouquier algebras are a family of graded algebras that have been introduced in \cite{KhovanovLauda} and \cite{Rouquier2KacMoody} in order to categorify the negative half of the quantum group $U_q(\lie)$ associated to a Kac-Moody algebra $\lie$. More specifically, there is a family of graded algebras $\{ \textbf{R}(n) \}_{n \in \mathbb{Z}_{\geq 0}}$ and a $\mathbb{Z}[q,q^{-1}]$-algebra isomorphism between the Grothendieck ring $K_0(\mathscr{C})$ and $U^-_q(\lie)_{\mathbb{Z}[q^{\pm 1}]}^\vee$, where $\mathscr{C}$ is the category of finite-dimensional graded $\textbf{R}(n)$-modules and $U^-_q(\lie)_{\mathbb{Z}[q^{\pm 1}]}^\vee$ is the integral form of the negative half of $U_q(\lie)$, for a Kac-Moody algebra $\lie$. Moreover, when $\lie$ is symmetric and the ground field is of characteristic 0, it has been shown that isomorphism classes of finite-dimensional self-dual simple modules correspond to upper global basis elements under this categorification (\cite{Rouquier}, \cite{CanonicalBasesAndKLR_Algebras}). The multiplication on $K(\mathscr{C})$ is given by the \textit{convolution product} and $q$ acts by shifting the grading. \\
\hspace*{4em} Let $\mathfrak{n}=\mathfrak{n}_+$ be the positive part of the Kac-Moody algebra $\lie$ in a triangular decomposition and let $\mathcal{A}_q(\mathfrak{n})$ be the associated quantum coordinate ring; a certain subalgebra of $U_q(\mathfrak{n})^\ast$. One can show that the algebras $\mathcal{A}_q(\mathfrak{n})$ and $U^-_q(\lie)$ are in fact isomorphic (see \cite{KKKOII}, Lemma 2.7). Furthermore for each $w \in W$, where $W$ is the Weyl group of $\lie$, $\mathcal{A}_q(\mathfrak{n})$ has a subalgebra denoted by $\mathcal{A}_q(\mathfrak{n}(w))$ and Gei\ss, Leclerc and Schr\"{o}er \cite{GeissLeclercSchroer} showed that, for every $w \in W$, $\mathcal{A}_q(\mathfrak{n}(w))$ admits a quantum cluster algebra structure. We briefly recall the idea of a cluster algebra. Cluster algebras form a class of commutative rings and were introduced by Fomin and Zelevinksy in a series of papers beginning with \cite{FominZelevinskyI}. They were originally introduced in order to study canonical bases of algebraic groups but have since appeared in many other areas of mathematics including Teichm\"{u}ller theory, Poisson geometry, and quiver representations. They are defined by giving a partial set of generators, called an \textit{initial cluster}, together with a process enabling one to obtain all other clusters and hence all generators. Berenstein and Zelevinksy subsequently introduced the definition of a quantum cluster algebra \cite{BerensteinZelevinsky} which also consists of providing an initial cluster $\{ x_1, \ldots, x_n \}$ together with an iterative process to obtain all generators. However, in this case the cluster variables now quantum-commute; $x_ix_j=q^{\lambda_{ij}}x_jx_i$, where $(\lambda_{ij})_{1\leq i,j\leq n}$ is a skew-symmetric $n\times n$ matrix with integer entries. Monomials in the variables of a quantum cluster $\{ x_1, \ldots, x_n \}$ are called \textit{(quantum) cluster monomials}. Gei\ss, Leclerc and Schr\"{o}er conjectured that, up to a power of $q^{\frac{1}{2}}$, quantum cluster monomials of $\mathcal{A}_{q^{\frac{1}{2}}}(\mathfrak{n}(w))$ are real global basis elements, i.e.\ if $b$ is an element of the upper global basis then $b^2$ is an element of the upper global basis. Using KLR algebras, and the work of Varagnolo-Vasserot and Rouquier stated above, this conjecture now states that quantum cluster monomials correspond to real simple self-dual modules, where a simple module $M$ over a KLR algebra is \textit{real} if the convolution product $M\circ M$ is a simple module. This is reminiscent of the definition of a monoidal categorification of a cluster algebra given by Hernandez and Leclerc in \cite{HernandezLeclerc}; let $C$ be a cluster algebra and let $\mathscr{A}$ be an abelian monoidal category. Then $\mathscr{A}$ is a \textit{monoidal categorification} of $C$ if there is a ring isomorphism $K_0(\mathscr{A})\cong C$ such that $(i)$ the cluster monomials of $C$ correspond to the classes of real simple objects of $\mathscr{A}$ and $(ii)$ the cluster variables of $C$ are the classes of all real prime simple objects of $\mathscr{A}$. Inspired by this Kang, Kashiwara, Kim and Oh introduced the notion of monoidal categorification of quantum cluster algebras \cite{KKKO}, \cite{KKKOII} and proved the conjecture. \\
\hspace*{4em} At the heart of their work, and of fundamental importance, are $R$-matrices for KLR algebras which they use in order to categorify the cluster mutation relations. Briefly, this means there exists a set of simple modules $\{M_k\}_{k \in J}$ in $\mathscr{C}_w$, where $\mathscr{C}_w$ is a certain subcategory of $\textbf{R}$-gmod, labelled by a finite index set $J=J_{\textrm{fr}}\sqcup J_{\textrm{ex}}$ which has frozen and exchangeable parts. This set will correspond to the initial cluster of the quantum cluster algebra $\mathcal{A}_q(\mathfrak{n}(w))$. We can mutate in direction $k\in J_{\textrm{ex}}$; we replace $M_k$ with a simple module $M_k^\prime$ which fits into two certain short exact sequences, using $R$-matrices, so that when we view the Grothendieck group we have a quantum cluster algebra type relation. The $R$-matrices for symmetric KLR algebras were constructed and studied in \cite{KKKI} using intertwining elements; certain elements of KLR algebras satisfying various properties such as the braid relations. $R$-matrices are solutions of the Yang-Baxter equation which we will now briefly recall. The Yang-Baxter equation, together with $R$-matrices, have applications in many areas of mathematics and mathematical physics including (quantum) integrable systems, statistical mechanics, knot theory and quantum groups. There are several variations of the Yang-Baxter equation but we describe the most elementary form here. Let $A$ be an associative unital algebra and let $R\in A\otimes A$ be an invertible element. The Yang-Baxter equation is
\begin{equation*} \label{equation:YB equation}
R_{12}R_{13}R_{23}=R_{23}R_{13}R_{12}
\end{equation*}
where $R_{12}=R\otimes 1$, $R_{23}=1\otimes R$ and  $R_{13}=(S\otimes 1)(1\otimes R)$, where $S:A\otimes A \longrightarrow A \otimes A, u\otimes v \mapsto v\otimes u$ is the switch map. An element $R \in A \otimes A$ satisfying the above equality is called an $R$-matrix. It is often the case that that $R$ is dependent upon a parameter $z$, often called a \textit{spectral parameter}, in which case $R=R(z)$ is called an \textit{affine $R$-matrix}. In this paper, $R$-matrices are homogeneous morphisms between convolution products of modules over KLR algebras and VV algebras which satisfy the Yang-Baxter equation and in fact we will also consider non-invertible $R$-matrices. Very recently progress has been made in various directions in areas involving $R$-matrices, including the work of Kang-Kashiwara-Kim-Oh described above and also work of Maulik and Okounkov \cite{OkounkovMaulik} for example. For an overview of this and other recent work, see \cite{HernandezBourbaki}. 
\\
\\
In 2010 Varagnolo and Vasserot introduced the VV algebras in order to prove a conjecture of Enomoto and Kashiwara \cite{EnomotoKashiwara2} which stated that affine Hecke algebras of type $B$ categorify an irreducible highest weight module $V_\theta(\lambda)$ over a certain algebra $\btg$. Moreover, they conjectured that $V_\theta(\lambda)$ admits a canonical basis and that this basis is in bijection with the set of isomorphism classes of certain simple modules over affine Hecke algebras of type $B$. Enomoto and Kashiwara proved this in a particular case before Varagnolo and Vasserot proved it in general using VV algebras. In particular they proved that categories of finite-dimensional modules over VV algebras are equivalent to categories of finite-dimensional modules over affine Hecke algebras of type $B$. They showed that there is a $\btg$-module isomorphism between the Grothendieck group of finitely generated graded projective VV algebra modules and $\vtl$. This differs to the categorification theorem for KLR algebras in which the Grothendieck group of finitely generated graded KLR algebra modules possesses a multiplication so that there is an algebra isomorphism between $K_0(\textbf{R}\textrm{-proj})$ and $U_q^-(\lie)_{\mathbb{Z}[q^{\pm 1}]}$. However, VV algebras are closely related to KLR algebras and in fact, under certain conditions, are Morita equivalent (see \cite{Walker}). 
\\
\\
The primary goal of this work is to show that the Grothendieck group of the category of graded VV algebra modules has a multiplication so that it is in fact an algebra. Furthermore, we construct this multiplication in such a way that it is compatible with the convolution product for KLR algebra modules so that we obtain a quantum cluster algebra. As an application of this work, we construct $R$-matrices between products of VV algebra modules which yield solutions to the Yang-Baxter equation and which exhibit particularly attractive properties with respect to certain simple modules. We view this work as a starting point in determining the existence of a quantum cluster algebra structure on a natural quotient of $\btg$. 
\\
\\
We begin by recalling the definitions of the KLR algebras in underlying type $A$ and the VV algebras in Section \ref{section:preliminaries}. In Section \ref{section:R-Matrices} we recall the convolution product and $R$-matrices for KLR algebra modules before describing the difficulties encountered when doing this for modules over VV algebras. Building on work in \cite{Walker} we then describe how to avoid these problems using an equivalence of categories and we illustrate this with a series of small examples. Finally, we define the normalized $R$-matrix for VV algebra modules and study its properties with respect to certain simple modules. We use examples throughout this work to calculate convolution products of modules and to demonstrate that our $R$-matrices can be calculated explicitly. In Section \ref{section:conjectures} we finish with a discussion on monoidal categorification of cluster algebras and we pose various questions surrounding VV algebras and discuss the possible directions in which this work could go. This work and the results have obvious parallels with the type $A$ setting, indeed the convolution product defined here is compatible with the convolution product between KLR algebra modules.  
\subsection*{Acknowledgements}
I thank David Hernandez for many helpful discussions concerning this work and for taking the time to read through a draft of this paper. I also thank Loic Poulain D'Andecy for his comments on an earlier version of this work.
\section{Preliminaries} \label{section:preliminaries}
In this paper $k$ will denote a field with char$(k)\neq 2$ and $k^\times$ will denote the set of non-zero elements of $k$. All modules will be left modules, all gradings will be $\mathbb{Z}$-gradings and $v$ will denote both a formal variable and a functor which shifts the degree by 1. So $vN$ is a graded $A$-module with $r^{\textrm{th}}$ graded component $(vN)_r=N_{r-1}$, for $N=\bigoplus_{i \in \mathbb{Z}} N_i$ a graded $A$-module. For a ring $A$ we write $A$-Mod to denote the category of $A$-modules and $A$-mod for the category of $A$-modules which are finite-dimensional as $k$-vector spaces. For a graded ring $A$ we write $A$-gMod to denote the category of graded $A$-modules and $A$-gmod for the category of finite-dimensional graded $A$-modules. 
\\
\\
We now go on to recall the defining relations of the KLR algebras and VV algebras. 
\subsection{KLR Algebras and VV Algebras} \label{section:vv definition}
\subsubsection*{KLR Algebras}
We first recall the definition of the KLR algebras in underlying type $A$, which were introduced both in \cite{KhovanovLauda} and in \cite{Rouquier2KacMoody}. The following notation and summary can be found in \cite{Walker} for example.
\\
\\
We start by fixing an element $p \in k^{\times}$ and considering the action of $\mathbb{Z}$ on $k^{\times}$ given by $n \cdot \lambda = p^{2n}\lambda$. Fix a $\mathbb{Z}$-orbit $\tilde{I}$ so $\tilde{I}=\tilde{I}_\lambda=\{ p^{2n}\lambda \mid n \in \mathbb{Z} \}$ is the $\mathbb{Z}$-orbit of $\lambda$. We associate to $\tilde{I}_\lambda$ a quiver $\tilde{\Gamma}= \tilde{\Gamma}_{\tilde{I}}$. The vertices of $\tilde{\Gamma}$ are the elements $i\in \tilde{I}$ and for every $i\in \tilde{I}$ we have an arrow $p^2i \longrightarrow i$. We exclude the cases $\pm 1 \in \tilde{I}$ and $p = \pm 1$. Now set $\mathbb{N}\tilde{I}:=\{ \tilde{\nu} = \sum_{i\in \tilde{I}}\tilde{\nu}_i i \mid \tilde{\nu}_i \in \mathbb{Z}_{\geq 0} \textrm{ and } |\supp(\tilde{\nu})| < \infty \hspace{0.2em} \}$, where $\supp(\tilde{\nu}):=\{ i \in \tilde{I} \mid \tilde{\nu}_i \neq 0 \}$. The height of $\tilde{\nu} \in \mathbb{N}\tilde{I}$ is defined to be $|\tilde{\nu}|:=\sum_{i \in \tilde{I}}\tilde{\nu}_i$. For $\tilde{\nu} \in \mathbb{N}\tilde{I}$ with $|\tilde{\nu}|=m$, we set $\tilde{I}^{\tilde{\nu}}:=\{ \textbf{i}=(i_1, \ldots, i_m) \in \tilde{I}^m \mid \sum_{k=1}^m i_k = \tilde{\nu} \}$.
\begin{defn} \label{definition:klr algebra and relations}
The \textbf{KLR algebra} associated to $\tilde{\nu}\in \mathbb{N}\tilde{I}$ with $|\tilde{\nu}|=m$ is denoted by $\textbf{R}_{\tilde{\nu}}$ and is the graded k-algebra generated by the elements 
\begin{equation*}
\{ x_k \}_{1\leq k\leq m} \cup \{ \sigma_l \}_{1\leq l<m} \cup \{ \idemp \}_{\textbf{i} \in \tilde{I}^{\tilde{\nu}}}
\end{equation*}
satisfying the following defining relations. 
\begin{enumerate}
\item $\textbf{e}(\textbf{i})\textbf{e}(\textbf{j}) = \delta_{\textbf{ij}}\textbf{e}(\textbf{i})$, \hspace{4 pt} $\sigma_{k}\textbf{e}(\textbf{i}) = \textbf{e}(s_{k}\textbf{i})\sigma_{k}$, \hspace{4 pt} $x_{l}\textbf{e}(\textbf{i})=\textbf{e}(\textbf{i})x_{l}$, \hspace{4 pt} $\sum_{\textbf{i} \in \tilde{I}^{\tilde{\nu}}} \idemp = 1$.
\item $x_rx_s=x_sx_r \textrm{ for all }1\leq r,s \leq m$.
\vspace*{0.5em}
\item $\sigma_{k}^{2}\textbf{e}(\textbf{i}) = Q_{i_k,i_{k+1}}(x_{k+1},x_k)\idmep$, \hspace{4 pt} $\sigma_{j}\sigma_{k}=\sigma_{k}\sigma_{j} \textrm{ for } j \neq k\pm 1$,\\ \\ \vspace*{0.5em}
$(\sigma_{k+1}\sigma_{k}\sigma_{k+1} - \sigma_{k}\sigma_{k+1}\sigma_{k})\idemp = \left\{
\begin{array}{ll}
\frac{Q_{i_k,i_{k+1}}(x_{k+1}, x_k)-Q_{i_k,i_{k+1}}(x_{k+1}, x_{k+2})}{x_k-x_{k+2}} & \textrm{ if }i_k = i_{k+2} \\
0 & \textrm{ if } i_k\neq i_{k+2}.
\end{array} \right.$ \vspace*{0.5em}
\item $(\sigma_k x_l - x_{s_k(l)}\sigma_k)\idemp=\left\{
\begin{array}{l l l}
			-\idemp & \quad \textrm{ if } l=k, i_k=i_{k+1}\\
			\idemp & \quad \textrm{ if } l=k+1, i_k=i_{k+1}\\
			0 & \quad \textrm{ else.}
\end{array} \right. $\\
\end{enumerate}
\end{defn}
The grading on $\textbf{R}_{\tilde{\nu}}$ is given via
\begin{equation*}
\begin{split}
&\textrm{deg}(\idemp)=0, \quad \textrm{deg}(x_l \idemp)=2, \\
&\textrm{deg}(\sigma_k \idemp)=\left\{
\begin{array}{l l l}	
|i_k \rightarrow i_{k+1}|+|i_{k+1} \rightarrow i_k| & \quad \textrm{ if  } i_k \neq i_{k+1} \\
-2 & \quad \textrm{ if  } 	i_k = i_{k+1}	
\end{array} \right. 
\end{split}
\end{equation*}
where 
\begin{equation*}
Q_{i,j}(x,y)=\left\lbrace
\begin{array}{ll}
(-1)^{|i\rightarrow j|}(x-y)^{|i\rightarrow j|+|j\rightarrow i|} & \textrm{ if }i\neq j \\
0 & \textrm{ if }i=j
\end{array}
\right.
\end{equation*}
and where $|i \rightarrow j|$ represents the number of arrows in $\tilde{\Gamma}$ which have origin $i$ and target $j$.
\\
\\
When $\tilde{\nu}=0$ we put $\textbf{R}_{\tilde{\nu}}=k$ as a graded $k$-algebra.
\\
\\
We also recall some standard notation. For each $w \in \sm$ fix a reduced expression. Let $w=s_{i_1} \cdots s_{i_r}$, where $1 \leq i_k < m$ for all $k$, be this reduced expression. We then define the element $\sigma_w$ by setting 
\begin{equation*}
\sigma_w= \sigma_{i_1}\sigma_{i_2} \cdots \sigma_{i_r}.
\end{equation*}
When $w=1$ is the identity in $\sm$ we have $\sigma_1\idemp=\idemp$ for all $\idemp$. Since reduced expressions of $w$ are not always unique $\sigma_w$ depends upon the choice of reduced expression of $w$.
\subsubsection*{VV Algebras} We now recall the family of graded algebras defined by Varagnolo and Vasserot in \cite{VaragnoloVasserot}. An extended version of this setup and notation can be found in \cite{Walker}. Fix invertible elements $p, q \in k$ such that $p$ is not a power of $q$ and vice versa. Consider the action of $\mathbb{Z} \rtimes \mathbb{Z}_{2}$ on $k^{\times}$ given by $(n,\varepsilon) \cdot \lambda = p^{2n}\lambda^{\varepsilon}$ and fix a $\mathbb{Z} \rtimes \mathbb{Z}_2$-orbit $I_\lambda=\{ p^{2n}\lambda^{\pm 1} \mid n \in \mathbb{Z} \}$. Let $I_\lambda$ be the set of vertices of a quiver $\Gamma=\Gamma_I$ in which, for every $i\in I$, there is an arrow $p^{2}i \longrightarrow i$. Define an involution $\theta$ on $\Gamma$ as follows. On vertices we have $\theta(i)=i^{-1}$ and on arrows we have
\begin{equation*}
\theta(p^{2}i \longrightarrow i)=p^{-2}i^{-1} \longleftarrow i^{-1}, \quad \textrm{ for all } i \in I .
\end{equation*}
We again exclude the cases $\pm 1 \in I$ and $p = \pm 1$ meaning that $\theta$ has no fixed points and $\Gamma$ has no loops (1-cycles). Now set ${^\theta}\mathbb{N}I:=\{ \nu = \sum_{i\in I}\nu_i i \mid \nu_i \in \mathbb{Z}_{\geq 0}, \nu_i= \nu_{\theta(i)} \hspace{0.5em} \forall i \textrm{ and } |\supp(\nu)|<\infty \}$. The height of $\nu \in {^\theta}\mathbb{N}I$ is defined to be $|\nu|=\sum_{i \in I}\nu_i$. Since we require $\nu_i=\nu_{\theta(i)}$ for every $i \in I$ it is always the case that $|\nu| \in 2\mathbb{Z}_{\geq 0}$. For $\nu \in {^\theta}\mathbb{N}I$, we set
\begin{equation*}
{^\theta}I^\nu:=\left\lbrace \textbf{i}=(i_1, \ldots, i_m) \in I^m : \nu= \sum_{k=1}^m i_k + \sum_{k=1}^m i_k^{-1}\right\rbrace.
\end{equation*}
\begin{remark} \label{remark:KLR algs from VV data}
We recall here that we are able to define KLR algebras with this data. More specifically, given a $\mathbb{Z} \rtimes \mathbb{Z}_2$-orbit $I_\lambda$ together with the quiver $\Gamma_{I_\lambda}$ we can pick $\tilde{\nu} \in \mathbb{N}I_\lambda$ which yields a KLR algebra $\subklr$. This is described in \cite{Walker}, Remark 1.11.
\end{remark}
\begin{defn} \label{definition:vv algebra}
The \textbf{VV algebra} $\vv$ associated to $\nu\in {^\theta}\mathbb{N}I$, $|\nu|=2m$ is the graded k-algebra generated by elements 
\begin{equation*}
\{ x_k \}_{1\leq k\leq m} \cup \{ \sigma_l \}_{1\leq l<m} \cup \{ \idemp \}_{\textbf{i} \in {^\theta}I^\nu} \cup \{ \pi \}
\end{equation*}
satisfying the relations given in Definition \ref{definition:klr algebra and relations} together with,
\begin{enumerate}
\item $\pi \textbf{e}(i_1,\ldots,i_m) = \textbf{e}(i_1^{-1},i_2 \ldots,i_m) \pi$, \\ $\pi x_{1} = -x_{1} \pi$, \\ $\pi x_{l} = x_{l} \pi$ \textrm{ if }$l > 1$, \\
$\pi \sigma_{k}=\sigma_{k} \pi$ \textrm{ if }$k\neq 1.$ \vspace*{0.5em}
\item $\pi^{2}\textbf{e}(\textbf{i})=\left\{
\begin{array}{l l l}
    x_1\idemp & \quad i_1=q \\
    -x_1\idemp & \quad i_1=q^{-1} \\
    \idemp & \quad i_1 \neq q^{\pm 1}.
\end{array} \right.$ \vspace*{0.5em}
\item $(\sigma_1 \pi)^2\idemp - (\pi\sigma_1)^2\idemp= \left\{
\begin{array}{l l l}
0 & \quad i_1\neq i_2^{-1}\textrm{ or if } i_1 \neq q^{\pm 1} \\
\sigma_1\idemp & \quad i_1=i_2^{-1}=q \\
-\sigma_1\idemp & \quad i_1=i_2^{-1}=q^{-1}
\end{array} \right. $
\end{enumerate}
\end{defn}
The grading on $\mathfrak{W}_{\nu}$ is defined as follows.
\begin{equation*}
\begin{split}
&\textrm{deg}(\idemp)=0, \quad \textrm{deg}(x_l \idemp)=2, \\
&\textrm{deg}(\pi\idemp)=\left\{
\begin{array}{l l}	
1 & \quad \textrm{ if  } i_1=q^{\pm 1} \\
0 & \quad \textrm{ if  } i_1\neq q^{\pm 1},	
\end{array} \right. \\
&\textrm{deg}(\sigma_k \idemp)=\left\{
\begin{array}{l l l}	
|i_k \rightarrow i_{k+1}|+|i_{k+1} \rightarrow i_k| & \quad \textrm{ if  } i_k \neq i_{k+1} \\
-2 & \quad \textrm{ if  } 	i_k = i_{k+1}.	
\end{array} \right. 
\end{split}
\end{equation*}
where $|i_{k} \rightarrow i_{k+1}|$ represents the number of arrows in the quiver $\Gamma$ which have origin $i_{k}$ and target $i_{k+1}$ . 
\\
\\
When $\nu=0$ we put $\vv=k$ as a graded $k$-algebra.
\\
\\
We denote the support of $\nu = \sum_{i \in I} \nu_i i \in \thetani$ by $\supp(\nu):=\{ i \in I \mid \nu_i \neq 0 \}$. Throughout this paper we denote by $W^B_m$ the Weyl group of type $B_m$. It has generators $s_0, s_1, \ldots, s_{m-1}$ which are subject to the relations
\begin{center}
\begin{tabular}{ll}
$s_i^2=1 \hspace*{0.7em}\forall i$, & $s_is_{i+1}s_i=s_{i+1}s_is_{i+1} \hspace*{0.3em} \textrm{ for } 1 \leq i \leq m-2$,\\
$s_0s_1s_0s_1=s_1s_0s_1s_0$, $\hspace*{2em}$ & $s_is_j=s_js_i \hspace*{0.3em} \textrm{ when } |i-j|>1.$
\end{tabular}
\end{center}
Finally, let us recall a \text{k}-basis of a VV algebra $\vv$. As with reduced expressions of elements of $\sm$ the same remarks made about KLR algebra elements apply to elements of $\vv$; for each element $w \in W^B_m$ we fix a reduced expression $w=s_{i_1}\cdots s_{i_r}$ so that $\sigma_w$ is well-defined. We set $\sigma_0:=\pi$.
\begin{lemma} \label{vv basis}
The following set is a k-basis for $\vv$, where $\nu \in {^\theta}\mathbb{N}I$ and $|\nu|=2m$. 
\begin{equation*}
\lbrace \sigma_w x_1^{n_1}\cdots x_m^{n_m}\idemp \mid w\in W^B_m, \textbf{i} \in {^\theta}I^\nu, n_k \in \mathbb{Z}_{\geq 0} \textrm{ } \forall k\rbrace.
\end{equation*}
\end{lemma}

\section{Convolution Products and $R$-Matrices}  \label{section:R-Matrices}
\subsection{The naive approach}\label{subsection:the naive approach}
In this section we recall how convolution products between modules over KLR algebras, together with $R$-matrices for KLR algebras, are defined and we attempt a similar construction for VV algebras. We find that in fact this initial approach does not work in general for VV algebras and we explain why this is. 
\\
\\
We start with the usual setup for VV algebras from which we can define KLR algebras, as outlined in Remark \ref{remark:KLR algs from VV data}. In particular, we have an orbit $I_\lambda=I$ and a corresponding quiver $\Gamma$. There exists a symmetric bilinear form $(\cdot,\cdot)$ on $\mathbb{N}I$ defined by
\begin{equation*}
(i,j)= \left\lbrace
\begin{array}{ll}
-|i \rightarrow j|-|j \rightarrow i| & \textrm{ if }i\neq j \\
2 & \textrm{ if }i=j
\end{array}
\right.
\end{equation*}
where, as before, $|i \rightarrow j|$ denotes the number of arrows in $\Gamma$ starting at $i$ and ending at $j$. Let $[\cdot,\cdot]$ be another symmetric bilinear form on $\mathbb{N}I$ defined by $[i,j]=\delta_{ij}$. Note then that we have $\textrm{deg}(\sigma_k\idemp)=-(i_k,i_{k+1})$, $\textrm{deg}(x_k\idemp)=(i_k,i_k)$ and $\textrm{deg}(\pi\idemp)=[i_1,q^{\pm 1}]$. 
\\
\\
Take $\alpha,\beta \in \mathbb{N}I$. For a graded $\textbf{R}_\alpha$-module $M$ and a graded $\textbf{R}_\beta$-module $N$ the convolution product $M\circ N$ is the $\textbf{R}_{\alpha+\beta}$-module given by
\begin{equation*}
M\circ N := \textbf{R}_{\alpha+\beta}e(\alpha,\beta)\underset{\textbf{R}_\alpha \otimes \textbf{R}_\beta}{\otimes}(M\otimes N)
\end{equation*}
where $e(\alpha,\beta):=\sum_{\substack{\textbf{i}\in I^\alpha \\ \textbf{j} \in I^\beta}}e(\textbf{i}\textbf{j})$ and the right-action of $\textbf{R}_{\alpha}\otimes \textbf{R}_\beta$ on $\textbf{R}_{\alpha+\beta}e(\alpha,\beta)$ is given via the algebra morphism 
\begin{equation*}
\phi:\textbf{R}_\alpha \otimes \textbf{R}_\beta \longrightarrow e(\alpha,\beta)\textbf{R}_{\alpha+\beta} e(\alpha,\beta).
\end{equation*}
We now recall the construction of $R$-matrices for KLR algebras. For $\alpha \in \mathbb{N}I$, $|\alpha|=m$ and $1\leq k <m$ we let $\varphi_k \in \textbf{R}_\alpha$ be defined by
\begin{equation} \label{equation:intertwiners}
\varphi_k\idemp= \left\lbrace
\begin{array}{ll}
(\sigma_kx_k-x_k\sigma_k)\idemp & \textrm{ if } i_k=i_{k+1} \\
\sigma_k\idemp & \textrm{ if }i_k\neq i_{k+1}.
\end{array}
\right.
\end{equation}
These elements of the KLR algebra are called \textit{intertwiners}. Furthermore we recall that, for $m,n \in \mathbb{Z}_{\geq 0}$, the element $w[m,n]\in\mathfrak{S}_{m+n}$ is given by
\begin{equation} \label{equation:the element w[m,n]}
w[m,n](k)=\left\lbrace
\begin{array}{ll}
k+n & \textrm{ if } 1\leq k \leq m \\
k-m & \textrm{ if } m<k\leq m+n.
\end{array}
\right.
\end{equation}
Then the $R$-matrix $R_{M,N}:M\circ N \longrightarrow q^{(\alpha,\beta)-2[\alpha,\beta]}N\circ M$ is a homogeneous $\textbf{R}_{\alpha+\beta}$-module morphism which is obtained by extending the linear map $M\otimes N\longrightarrow q^{(\alpha,\beta)-2[\alpha,\beta]}N\circ M$, $u\otimes v \mapsto \varphi_{w[n,m]}(v\otimes u)$. Note that $w[m,n] \in \mathfrak{S}_{m+n}$ is the longest minimal length left coset representative of $\sm\times\sn$ in $\mathfrak{S}_{m+n}$.
\\
\\
In order to work with VV algebras we must specify whether the parameters $p$ and $q$ lie in $I$. For now, and for the remainder of this paper, we fix the following setting.
\begin{center}
\begin{tcolorbox}[width=2in,
                  boxsep=0pt,
                  left=0pt,
                  right=0pt,
                  top=0.6em]
\begin{center}
\textbf{We assume $\mathbf{\pm p}$, $\mathbf{\pm q\not\in I}$.}
\end{center}
\end{tcolorbox}
\end{center}
In particular this means $\textrm{deg}(\pi\idemp)=0$ for every $\textbf{i} \in \thetainu$. \\
\\
The naive method for constructing convolution products of VV algebra modules would be to mimic the above construction, as follows. Take $\beta, \gamma \in \thetani$, with $|\beta|=2m$ and $|\gamma|=2n$, and let $M$ be a graded $\w_\beta$-module and $N$ a graded $\w_\gamma$-module. Then set
\begin{equation*}
M\circ N:=\w_{\beta+\gamma}e(\beta,\gamma) \underset{\w_\beta\otimes \w_\gamma}{\otimes}(M\otimes N)
\end{equation*}
where $e(\beta,\gamma):=\sum_{\substack{\textbf{i}\in {^\theta}I^\beta \\ \textbf{j} \in {^\theta}I^\gamma}}e(\textbf{i}\textbf{j})$ and where the right-action of $\vvb \otimes \vvg$ on $\vvbg e(\beta,\gamma)$ is defined in a similar way to the KLR algebra setting. However, the map 
\begin{equation*}
\phi:\w_\beta \otimes \w_\gamma \longrightarrow e(\beta,\gamma)\w_{\beta+\gamma} e(\beta,\gamma)
\end{equation*}
is an algebra morphism only when $\supp(\beta)\cap \supp(\gamma)=\emptyset$ and when there are no arrows, in the underlying quiver $\Gamma$, between elements of $\supp(\beta)$ and $\supp(\gamma)$. We demonstrate this explicitly with the following example. We note that under this map $\phi$ we have $\idemp \otimes \pi\idempj \mapsto \sigma_{m}\cdots \sigma_1\pi \sigma_1 \cdots \sigma_{m}\idempij$. 
\begin{example}
Take $\beta=\lambda+\lambda^{-1}$, $\gamma=p^2\lambda +p^{-2}\lambda^{-1}$ and note that there is an arrow $p^2\lambda \longrightarrow \lambda$ in $\Gamma$. Consider $\phi: \w_\beta \otimes \w_\gamma \longrightarrow e(\beta,\gamma)\w_{\beta+\gamma}e(\beta,\gamma)$. On the one hand we have 
\begin{equation*}
\phi(e(\lambda)\otimes e(p^{-2}\lambda^{-1}))=e(\lambda,p^{-2}\lambda^{-1}).
\end{equation*}
But on the other hand we also have
\begin{equation*}
\begin{split}
\phi(e(\lambda)\otimes e(p^{-2}\lambda^{-1}))&=\phi(e(\lambda)\otimes \pi^2e(p^{-2}\lambda^{-1})) \\
&=\phi(e(\lambda)\otimes \pi e(p^2\lambda))\phi(e(\lambda)\otimes \pi e(p^{-2}\lambda^{-1})) \\
&=\sigma_1\pi\sigma_1 e(\lambda,p^2\lambda)\cdot\sigma_1\pi\sigma_1e(\lambda,p^{-2}\lambda^{-1}) \\
&=-(x_1+x_2)e(\lambda,p^{-2}\lambda^{-1}).
\end{split}
\end{equation*}
\end{example}
Similarly one can show that $\phi$ is not a morphism when $\supp(\beta)\cap\supp(\gamma)\neq \emptyset$. Therefore, in general, one does not have a right-action of $\w_\beta \otimes \w_\gamma$ on $\w_{\beta+\gamma}e(\beta,\gamma)$ and the convolution product given above is not well-defined. In the case where one does have an algebra morphism, i.e.\ when $\supp(\beta)\cap \supp(\gamma)=\emptyset$ and there are no arrows between $\supp(\beta)$ and  $\supp(\gamma)$, $\phi$ is in fact an isomorphism (\cite{Walker}, Proposition 2.6). So, until stated otherwise, let us impose these (very strong) restrictions in order to have a well-defined product.
\vspace{0.3em}
\begin{center}
\textbf{We assume both $\supp(\beta)\cap \supp(\gamma)=\emptyset$ \\ and $\supp(\beta) \nleftrightarrow \supp(\gamma)$}
\end{center}
\vspace{0.3em}
where $\supp(\beta) \nleftrightarrow \supp(\gamma)$ indicates that there are no arrows, in the underlying quiver $\Gamma$, between elements of $\supp(\beta)$ and $\supp(\gamma)$. In the following section we will find a method which allows us to drop these restrictions. Take $\beta, \gamma \in \thetani$, with ht$(\beta)=2m$ and ht$(\gamma)=2n$, in such a way that there are no arrows between elements of $\supp(\beta)$ and $\supp(\gamma)$. Then the map $\phi: \w_\beta \otimes \w_\gamma \longrightarrow e(\beta,\gamma)\w_{\beta+\gamma}e(\beta,\gamma)$ is an algebra morphism and we can define the convolution product $M\circ N$ as above, where $M\in \vvb\textrm{-mod}$ and $N\in\vvg\textrm{-mod}$. This product is associative.
\begin{lemma}\label{lemma:M circ N decomposition}
Take $M\in \vvb\textrm{-Mod}$ and $N\in \vvg\textrm{-Mod}$. Let $\mathfrak{S}_{m,n}$ denote the minimal length left coset representatives of $\sm\times \sn$ in $\mathfrak{S}_{m+n}$. Then
\begin{equation*}
M\circ N = \bigoplus_{w \in \mathfrak{S}_{m,n}} \sigma_w(M\otimes N).
\end{equation*}
\end{lemma}
\begin{proof}
From \cite{ChiKinMak}, Proposition 3.3.1 we know that there exist minimal length left coset representatives of the quasi-parabolic subgroup $W^B_m\times W^B_n$ in $W^B_{m+n}$. Denote the set consisting of these elements by $W^B_{m,n}$. Every $x\in W^B_{m+n}$ can be written in the form $x=wy$, where $y \in W^B_m\times W^B_n$ and $w \in W^B_{m,n}$. Fix such a reduced expression for every $x \in W^B_{m+n}$. Then, since there are no arrows between $\supp(\beta)$ and $\supp(\gamma)$, we have $\sigma_x=\sigma_w\sigma_y$. By considering the basis elements of the VV algebra $\vvbg$ it is clear that $M\circ N$ has a basis consisting of elements $\sigma_w(m\otimes n)$, where $m, n$ are basis elements of $M,N$ respectively and $w \in W^B_{m,n}$. Furthermore for $w,w^\prime \in W^B_{m,n}$ such that $w\neq w^\prime$ we have $\sigma_w\textbf{e}(\beta,\gamma)=\idempj\sigma_w$, $\sigma_{w^\prime}\textbf{e}(\beta,\gamma)=\textbf{e}(\textbf{j}^\prime)\sigma_{w^\prime}$ with $\textbf{j}\neq \textbf{j}^\prime$ so that we obtain a direct sum. We finish by noting that $W^B_{m,n}=\smcomman$.
\end{proof}
\begin{remark}
We note that for $w \in \mathfrak{S}_{m,n}$, since there are no arrows between $\supp(\beta)$ and $\supp(\gamma)$, the element $\sigma_we(\beta,\gamma) \in \vvbg$ does not depend on the choice of reduced expression of $w$.
\end{remark}
Following \cite{KKKI} we would like intertwining elements $\varphi_0,\varphi_1,\ldots, \varphi_{n-1}$ of VV algebras in order to define $R$-matrices. The intertwiners $\varphi_1,\ldots, \varphi_{n-1}$ have already been defined for KLR algebras, see (\ref{equation:intertwiners}), so it remains for us to define $\varphi_0$. For every $\textbf{i} \in \thetainu$ we put
\begin{equation*}
\varphi_0\idemp:=\pi\idemp.
\end{equation*}
Then we have the following Lemma.
\begin{lemma}\label{lemma:varphi satisfies the braid rels} \leavevmode
\begin{enumerate}
\item The elements $\{ \varphi_a \}_{0 \leq a < m }$ satisfy the type $B$ braid relations.
\item Fix a reduced expression $w=s_{a_1}\cdots s_{a_k}$ of $w \in W^B_m$ and put $\varphi_w=\varphi_{a_1}\cdots \varphi_{a_k}$. Then $\varphi_w$ is independent of the choice of reduced expression of $w$.
\item For $w \in W^B_m$ and $1 \leq k \leq m$, we have $\varphi_wx_k=x_{w(k)}\varphi_w$ (where we define $x_{1-l}=-x_l$).
\item For $w \in W^B_m$ and $1\leq k<m$, if $w(k+1)=w(k)+1$ then $\varphi_w\sigma_k=\sigma_{|w(k)|}\varphi_w$.
\end{enumerate}
\end{lemma}
\begin{proof}
From \cite{VaragnoloVasserot}, Proposition 7.4 and Corollary 7.6 we have a polynomial representation of $\vv$ which is faithful. The action of the $\{ \varphi_a \}_{0\leq a < m}$ in this representation is the same as the action of the type $B$ Weyl group generators (possibly with a polynomial factor). It follows that, since the representation is faithful, the $\varphi_a$ satisfy the braid relations. So $(1)$ is proved and then $(2)$ follows immediately. In fact, $(3)$ also follows from the faithful polynomial representation. Finally, $(4)$ follows from the polynomial representation together with the fact that if $w(k+1)=w(k)+1$ then $ws_k=s_{|w(k)|}w$, for $w\in W^B_m$ and $1\leq k \leq m$.
\end{proof}
For $w \in W^B_m$, the element $\varphi_w\idemp$ has degree 
\begin{equation*}
\textrm{deg}(\varphi_w\idemp)=\sum_{\substack{1\leq a<b\leq m \\ w(a)>w(b)}} (-(i_a,i_b)+2[i_a,i_b])+\sum_{\substack{1\leq a \\ w(a)<1}}[i_a,q^{\pm 1}].
\end{equation*}
Since we are assuming $q \not\in I$ we have $\sum_{\substack{1\leq a \\ w(a)<1}}[i_a,q^{\pm 1}] =0$ and so
\begin{equation*}
\textrm{deg}(\varphi_w\idemp)=\sum_{\substack{1\leq a<b\leq m \\ w(a)>w(b)}} (-(i_a,i_b)+2[i_a,i_b]).
\end{equation*}
For $m,n\in \mathbb{Z}_{\geq 0}$ let $w[m,n]$ be the element of $\mathfrak{S}_{m+n}$ defined by (\ref{equation:the element w[m,n]}) so that $\textrm{deg}(\varphi_{w[n,m]})=2[\beta,\gamma]-(\beta,\gamma)$. Let $M$ be a $\vvb$-module with ht$(\beta)=2m$ and $N$ a $\vvg$-module with ht$(\gamma)=2n$. Using Lemma \ref{lemma:varphi satisfies the braid rels}, we have a $\vvb \otimes \vvg$-linear map $M\otimes N \longrightarrow N\circ M$ given by $u\otimes v \mapsto \varphi_{w[n,m]}(v\otimes u)$. We can make this map homogeneous by shifting the grading of $N\circ M$ by $q^{(\beta,\gamma)-2[\beta,\gamma]}$. This can be extended to a $\vvbg$-module morphism
\begin{equation*}
R_{M,N}:M\circ N \longrightarrow N \circ M
\end{equation*}
where we have omitted the degree shift. Then we obtain commutative diagrams
\begin{center}
\begin{tabular}{cc}
\begin{minipage}{7cm}
\begin{tikzpicture}
\matrix (m) [matrix of math nodes, row sep=4em,
column sep=5em, text height=1.5ex, text depth=0.25ex]
{L\circ M\circ N & M\circ L\circ N \\
  & M\circ N\circ L \\};
\path[-stealth]
(m-1-1) edge node [above] {$R_{L,M}$} (m-1-2)
        edge node [left] {$R_{L,M\circ N}$} (m-2-2)
(m-1-2) edge node [right] {$R_{L,N}$} (m-2-2);
\end{tikzpicture}
\end{minipage}
&
\begin{minipage}{7cm}
\begin{tikzpicture}
\matrix (m) [matrix of math nodes, row sep=4em,
column sep=5em, text height=1.5ex, text depth=0.25ex]
{L\circ M\circ N & L\circ N\circ M \\
  & N\circ L\circ M \\};
\path[-stealth]
(m-1-1) edge node [above] {$R_{M,N}$} (m-1-2)
        edge node [left] {$R_{L\circ M,N}$} (m-2-2)
(m-1-2) edge node [right] {$R_{L,N}$} (m-2-2);
\end{tikzpicture}
\end{minipage}
\end{tabular}
\end{center}
and the maps satisfy the Yang-Baxter equation; $R_{M,N}R_{L,N}R_{L,M}=R_{L,M}R_{L,N}R_{M,N}$.
\\
\\
Take $\beta,\gamma \in \thetani$ of heights $m,n$ respectively and let $M \in \vvb\textrm{-Mod}$, $N\in \vvg\textrm{-Mod}$. We have the following lemma under our assumptions.
\begin{lemma} \label{claim:M and N commute}
When $\supp(\beta)\cap \supp(\gamma)=\emptyset$ and $\supp(\beta) \nleftrightarrow \supp(\gamma)$ we have the following.
\begin{itemize}
\item[(i)] The $R$-matrix $R_{M,N}:M\circ N \longrightarrow N\circ M$ is non-zero.
\item[(ii)] We always have $M\circ N \cong N\circ M$.
\end{itemize}
\end{lemma}
\begin{proof}
We first note that the element $w[n,m]$ is the longest minimal length left coset representative of $\sn \times \sm$ in $\mathfrak{S}_{n+m}$. Then $\tau(w[n,m])$ is the longest minimal length left coset representative of $\sm \times \sn$ in $\mathfrak{S}_{m+n}$, where $\tau$ is the anti-involution which maps $s_i \mapsto s_i$ for all $i$. Equivalently, $w[n,m]$ is the longest minimal length right coset representative of $\sm \times \sn$ in $\mathfrak{S}_{m+n}$. It has the form
\begin{equation*}
w[n,m]=s_m\cdots s_2s_1 \cdots \cdots s_{m+n-2}\cdots s_ns_{n-1}s_{n+m-1}\cdots s_{n+1}s_n.
\end{equation*}
Then all elements of $\smcomman$ are prefixes of $\tau(w[n,m])$, where a prefix of an element $s_{i_1}\cdots s_{i_n}$ is of the form $s_{i_k}\cdots s_{i_n}$ for some $k$, $1\leq k \leq n$. We also remark that, since $\supp(\beta)\cap \supp(\gamma)=\emptyset$, we have $\varphi_{w[n,m]}=\sigma_{w[n,m]}$ and this is independent of the choice of reduced expression of $w[n,m]$. Then $R_{M,N}$ is given by
\begin{equation*}
\begin{split}
R_{M,N}:M\circ N &\longrightarrow N\circ M \\
\sigma_w(p\otimes q) &\mapsto \sigma_w\sigma_{w[n,m]}(q\otimes p)
\end{split}
\end{equation*}
where $w$ is a prefix of $\tau(w[n,m])$. Finally since $\supp(\beta) \nleftrightarrow \supp(\gamma)$ and $\supp(\beta)\cap \supp(\gamma)=\emptyset$ it follows that $\sigma_w\sigma_{w[n,m]}(q\otimes p) \neq 0$ and so $(i)$ is proved.
\\
\\
For $(ii)$ we first note that for each $w\in \smcomman$ there is an element $w\cdot w[n,m] =y \in \sncommam$ and moreover every element in $\sncommam$ can be obtained in this way. Then we have $\sigma_w\sigma_{w[n,m]}=\sigma_y$ and $R_{M,N}$ is bijective.
\end{proof}
\begin{lemma}
Suppose $\supp(\beta)\cap \supp(\gamma)=\emptyset$, $\supp(\beta) \nleftrightarrow \supp(\gamma)$. If $M \in \vvb\textrm{-Mod}$ and $N \in \vvg\textrm{-Mod}$ are simple modules then $M\circ N$ is a simple $\vvbg$-module.
\end{lemma}
\begin{proof}
We know from Lemma \ref{lemma:M circ N decomposition} that $M\circ N=\bigoplus_{w\in \mathfrak{S}_{m,n}} \sigma_w(M\otimes N)$. Since there are no arrows between elements of $\textrm{supp}(\beta)$ and $\textrm{supp}(\gamma)$ we know that for each $\sigma_w\idemp$, $w\in \smcomman$, there exists $v \in \smn$ such that $\sigma_v\sigma_w\idemp=\idemp$.
\end{proof}
As mentioned in the introduction, real simple modules will be essential if we hope to establish a quantum cluster algebra structure on Grothendieck rings of subcategories consisting of modules over VV algebras. But, since we are omitting cases when $\supp(\beta) \cap \supp(\gamma) \neq \emptyset$, the notion of a real simple module does not make sense here (since in this case we have $\beta=\gamma$).
\subsection{Examples}
Here we have a collection of examples of convolution products and $R$-matrices, starting with a KLR algebra example before several VV algebra examples in order to make explicit the discussion thus far.
\begin{notation*}\label{notation}
When $p \not\in I_\lambda$, the quiver $\Gamma_{I_\lambda}$ has two disjoint branches; one with vertex set $I_\lambda^+=\{ p^{2k}\lambda\mid k \in \mathbb{Z} \}$ and the other with vertex set $I_\lambda^-=\{ p^{2k}\lambda^{-1}\mid k \in \mathbb{Z} \}$. In most of the examples in this paper we will write $2k$ for the vertex $p^{2k}\lambda \in I^+_\lambda$ and $\overbar{2k}$ for the vertex $p^{2k}\lambda^{-1} \in I^-_\lambda$.
\end{notation*}
\begin{example}[for KLR algebras]\label{example:type A example}
Consider the underlying type $A$ case. We have $I=I_\lambda$ and the underlying quiver $\Gamma_I$ is
\begin{center}
\begin{tikzpicture}
\matrix (m) [matrix of math nodes, row sep=2.5em,
column sep=2em, text height=1.5ex, text depth=0.25ex]
{ \cdots & p^4\lambda & p^2\lambda & \lambda & p^{-2}\lambda & p^{-4}\lambda & \cdots \\};
\path[->]
(m-1-1) edge  (m-1-2)
(m-1-2) edge  (m-1-3)
(m-1-3) edge  (m-1-4)
(m-1-4) edge  (m-1-5)
(m-1-5) edge  (m-1-6)
(m-1-6) edge  (m-1-7);
\end{tikzpicture}
\end{center}
Following the notational comment above, we write $0$ instead of $\lambda$ and $4$ instead of $p^4\lambda$. Let $\beta=\lambda, \gamma=p^4\lambda\in \mathbb{N}I$. Let $M=\textbf{R}_\beta e(0) /\textrm{rad}(\textbf{R}_\beta e(0)) \in \textbf{R}_\beta$-mod which, as a k-vector space, is 1-dimensional with basis $\{ m=e(0) \}$. Let $N=\textbf{R}_\gamma e(4) /\textrm{rad}(\textbf{R}_\gamma e(4)) \in \textbf{R}_\gamma$-mod which is also 1-dimensional with basis $\{ n=e(4) \}$. Then, as vector spaces, one has
\begin{equation*}
\begin{split}
&M\circ N = \big\langle e(04)(m \otimes n), \sigma_1(m\otimes n) \big\rangle_k \\
&N\circ M = \big\langle e(40)(n \otimes m), \sigma_1(n\otimes m) \big\rangle_k.
\end{split}
\end{equation*}
We now compute the $R$-matrix $R_{M,N}:M\circ N \longrightarrow N\circ M$, noting that $\varphi_{w[1,1]}=\varphi_1=\sigma_1$ in this example.
\begin{equation*}
\begin{split}
R_{M,N}(e(04)(m \otimes n))&=\sigma_1(n \otimes m) \\
R_{M,N}(\sigma_1(m\otimes n))&= e(40)(n \otimes m).
\end{split}
\end{equation*}
So we have 
\begin{equation*}
R_{M,N}=\begin{pmatrix}
0 & 1 \\ 1 & 0
\end{pmatrix}.
\end{equation*}
Here, $R_{M,N}$ is invertible and $M\circ N \cong N \circ M$ demonstrating Lemma \ref{claim:M and N commute} applied to KLR algebras. 
\end{example}
\begin{example}[for VV algebras] \label{example:vv algebra r-matrix 1}
We recall that we are in the setting in which the underlying quiver $\Gamma$ has the form
\begin{center}
\begin{tikzpicture}
\matrix (m) [matrix of math nodes, row sep=2.5em,
column sep=2em, text height=1.5ex, text depth=0.25ex]
{ \cdots & p^2\lambda & \lambda & p^{-2}\lambda & \cdots \\
  \cdots & p^{-2}\lambda^{-1} & \lambda^{-1} & p^2\lambda^{-1} & \cdots \\};
\path[->]
(m-1-1) edge  (m-1-2)
(m-1-2) edge  (m-1-3)
(m-1-3) edge  (m-1-4)
(m-1-4) edge  (m-1-5)
(m-2-5) edge  (m-2-4)
(m-2-4) edge  (m-2-3)
(m-2-3) edge  (m-2-2)
(m-2-2) edge  (m-2-1);
\end{tikzpicture}
\end{center}
Let us look at the type $B$ analogue of Example \ref{example:type A example} above. Namely, take $\beta=\lambda+\lambda^{-1}$ and $\gamma=p^4\lambda+p^{-4}\lambda^{-1}$. Let 
\begin{equation*}
M=\w_\beta e(0)/ \textrm{rad}(\w_\beta e(0)) \in \w_\beta\textrm{-mod}. 
\end{equation*}
This is a simple module. As a k-vector space it has basis $\{ e(0),\pi e(0) \}$. Similarly, let 
\begin{equation*}
N=\w_\gamma e(4) / \textrm{rad}(\w_\gamma e(4)) \in \w_\gamma\textrm{-mod.} 
\end{equation*}
Then $N$ is also a simple module and has k-basis $\{ e(4),\pi e(4) \}$. Then $M\circ N$ is 8-dimensional with k-basis
\begin{equation*}
\left\lbrace
\begin{tabular}{llll} 
$e(e(0) \otimes e(4))$ & $e(\pi e(0) \otimes e(4))$ & $e(e(0)\otimes \pi e(4))$ & $e(\pi e(0) \otimes \pi e(4))$ \\
$\sigma_1 (e(0) \otimes e(4))$ & $\sigma_1(\pi e(0) \otimes e(4))$ & $\sigma_1(e(0)\otimes \pi e(4))$ & $\sigma_1(\pi e(0) \otimes \pi e(4))$ \\
\end{tabular} \right\rbrace.
\end{equation*}
Similarly, $N \circ M$ is 8-dimensional, with k-basis
\begin{equation*}
\left\lbrace
\begin{tabular}{llll} 
$e(e(4) \otimes e(0))$ & $e(\pi e(4) \otimes e(0))$ & $e(e(4)\otimes \pi e(0))$ & $e(\pi e(4) \otimes \pi e(0))$ \\
$\sigma_1 (e(4) \otimes e(0))$ & $\sigma_1(\pi e(4) \otimes e(0))$ & $\sigma_1(e(4)\otimes \pi e(0))$ & $\sigma_1(\pi e(4) \otimes \pi e(0))$ \\
\end{tabular} \right\rbrace.
\end{equation*}
Here we have $\varphi_{w[1,1]}=\varphi_1=\sigma_1$ and the $R$-matrix $R_{M,N}$ is given by
\begin{equation*} R_{M,N}=
\setcounter{MaxMatrixCols}{20}
\begin{pmatrix} 
0&1&0&0&0&0&0&0 \\ 
1&0&0&0&0&0&0&0 \\ 
0&0&0&0&0&1&0&0 \\ 
0&0&0&0&1&0&0&0 \\ 
0&0&0&1&0&0&0&0 \\ 
0&0&1&0&0&0&0&0 \\ 
0&0&0&0&0&0&0&1 \\ 
0&0&0&0&0&0&1&0
\end{pmatrix}
\end{equation*}
and is invertible so that $M\circ N \cong N\circ M$. We see that, up to a reordering of basis elements, this is just a larger version of Example \ref{example:type A example} as one might expect.
\\
\\
We look at another type $B$ example.
\end{example}
\begin{example}[for VV algebras] 
Take $\beta=\lambda+\lambda^{-1}+p^2\lambda+p^{-2}\lambda^{-1}$ and $\gamma=p^6\lambda+p^{-6}\lambda^{-1}$. We use the notation from \ref{notation} so we write $0$ instead of $\lambda$ and $2$ instead of $p^2\lambda$. Let
\begin{equation*}
M=\w_\beta e(02)/ \textrm{rad}(\w_\beta e(02)) \in \w_\beta\textrm{-mod}.
\end{equation*}
This is a 4-dimensional simple $\w_\beta$-module with k-basis $\{ e(02), \pi e(02), \sigma_1\pi e(02), \pi\sigma_1\pi e(02) \}$. Let
\begin{equation*}
N=\w_\gamma e(6)/\textrm{rad}(\w_\gamma e(6)) \in \w_\gamma\textrm{-mod}.
\end{equation*}
It is a 2-dimensional simple module with basis $\{ e(6), \pi e(6) \}$. \\
\\
Then one can check that both $M\circ N$ and $N\circ M$ are 24 dimensional and are in fact isomorphic.
\end{example}
\subsection{Induction and Restriction}
From the discussion above we see that requiring $\supp(\beta)\cap\supp(\gamma)=\emptyset$ and no arrows to exist between $\supp(\beta)$ and $\supp(\gamma)$ is too strong a condition to impose. For example, the notion of a real simple module does not exist in this case. So we seek to find an alternative construction for convolution products in $\w\textrm{-Mod}$. 
\\
\\
Since we have assumed $p,q \not\in I$ we have a Morita equivalence between $\klrplus$ and $\vv$ (\cite{Walker}, Theorem 2.10). In other words, there is an equivalence of categories $\klrplus\textrm{-Mod} \sim \vv\textrm{-Mod}$. Given $M \in \vvb\textrm{-Mod}$ and $N \in \vvg\textrm{-Mod}$ let us redefine the convolution product by first restricting to the KLR algebra module category, performing the usual convolution product, before returning to the VV algebra module category, i.e.\ we set $M \circ N := \textrm{Ind}(\textrm{Res}(M)\circ\textrm{Res}(N))$ where $\textrm{Res}(M)\circ\textrm{Res}(N)$ is the convolution product in the KLR algebra module category. This leads us to the following definition.
\begin{defn}
The convolution product of $M \in \vvb\textrm{-Mod}$, $N \in \vvg\textrm{-Mod}$ is given by
\begin{equation*}
M \circ N := \w_{\beta+\gamma}e(\beta^+\gamma^+) \underset{\textbf{R}_{\beta^+}\otimes\textbf{R}_{\gamma^+}}{\otimes} (\Res(M)\otimes \Res(N)).
\end{equation*}
\end{defn}
For $\alpha\in \thetani$ of height $2m$ we write $\alpha^+ \in \mathbb{N}I^+$ for the positive part, which has height $m$. Consider the induction functor $\vva\textbf{e}\otimes_{\textbf{e}\vva\textbf{e}} (-)$, where $\textbf{e}=\sum_{\textbf{i}\in I^{\alpha^+}} \idemp$ is the Morita idempotent. Let us write $R(-)=\textbf{e}\cdot -$ for restriction, noting that this is not to be confused with the notation for an $R$-matrix. One has,
\begin{equation*}
\begin{split}
\textrm{Ind}(\textrm{Res}(M)\circ\textrm{Res}(N)) &=\w_{\beta+\gamma}\textbf{e} \underset{\textbf{e}\w\textbf{e}}{\otimes} (R(M) \circ R(N)) \\ &=\w_{\beta+\gamma}\textbf{e} \underset{\textbf{e}\w\textbf{e}}{\otimes} \Big(\textbf{R}_{\beta^++\gamma^+}e(\beta^+\gamma^+)\underset{\textbf{R}_{\beta^+}\otimes \textbf{R}_{\gamma^+}}{\otimes} (R(M) \otimes R(N))\Big) \\
&\cong \w_{\beta+\gamma}\textbf{e}\w_{\beta+\gamma}e(\beta^+\gamma^+)\underset{\textbf{R}_{\beta^+}\otimes \textbf{R}_{\gamma^+}}{\otimes} (R(M) \otimes R(N)) \\
&\cong \w_{\beta+\gamma}e(\beta^+\gamma^+)\underset{\textbf{R}_{\beta^+}\otimes \textbf{R}_{\gamma^+}}{\otimes} (R(M) \otimes R(N)).
\end{split}
\end{equation*}
We have used the fact that $\textbf{R}_{\beta^++\gamma^+} \cong \textbf{e}\vvbg\textbf{e}$, where $\textbf{e}=\sum_{\textbf{i}\in I^{\beta^++\gamma^+}}\idemp$, in the third line and have used that $\textbf{e}$ is full in $\w_{\beta+\gamma}$ for the final isomorphism.
\\
\\
Note that this is indeed well-defined since there exists an algebra morphism
\begin{equation*}
\chi:\textbf{R}_{\beta^+} \otimes \textbf{R}_{\gamma^+} \longrightarrow e(\beta^+\gamma^+)\w_{\beta+\gamma}e(\beta^+\gamma^+)
\end{equation*}
which is given in the obvious way. 
\\
\\
Let $\mathcal{D}(W^B_{m+n}/(\sm \times \sn))$ denote the minimal length left coset representatives of $\sm \times \sn$ in $W^B_{m+n}$.
\begin{lemma}\label{lemma:the form of the conv prod}
Take $\beta, \gamma \in \thetani$ with $|\beta|=2m$ and $|\gamma|=2n$. For $M\in \vvb\textrm{-Mod}$ and $N \in \vvg\textrm{-Mod}$,
\begin{equation*}
M\circ N=\bigoplus_{w \in \mathcal{D}(W^B_{m+n}/(\sm \times \sn))} \sigma_w(R(M) \otimes R(N)).
\end{equation*}
\end{lemma}
\begin{proof}
This is clear by looking at the bases of VV algebras and KLR algebras.
\end{proof}
\begin{lemma}
The convolution product for VV algebra modules is associative. 
\end{lemma}
\begin{proof} 
This follows from the associativity of the convolution product for KLR algebra modules, shown below. We write $\circ_B$, $\circ_A$ for the convolution product of VV, KLR algebra modules respectively. Furthermore, we write $I$ for induction and $R$ for restriction. 
\\ 
\\
Suppose we have $\alpha,\beta,\gamma \in \thetani$ and $L\in \w_\alpha\textrm{-Mod}$, $M\in \w_\beta\textrm{-Mod}$, $N\in \w_\gamma\textrm{-Mod}$.
\begin{equation*}
\begin{split}
(M\circ_B L)\circ_B N&=I(R(M)\circ_A R(L))\circ_B N \\
&=I(RI(R(M)\circ_A R(L))\circ_A R(N)) \\
&=I((R(M)\circ_A R(L))\circ_A R(N))\\
&=I(R(M)\circ_A (R(L)\circ_A R(N)))\\
&=I(R(M)\circ_A RI(R(L)\circ_A R(N))) \\
&=M \circ_B I(R(L)\circ_A R(N)) \\
&=M\circ_B (L\circ_B N).
\end{split}
\end{equation*}
\end{proof}
We also remark that the convolution product for VV algebra modules is, by construction, compatible with the convolution product for KLR algebra modules in the sense of the following two lemmas. Let us keep the notation from the previous lemma.
\begin{lemma} Let $M \in \vva\textrm{-Mod}$, $N \in \vvb\textrm{-Mod}$ be VV algebra modules and, $P \in \textbf{R}_\gamma\textrm{-Mod}$ and $Q \in \textbf{R}_\delta\textrm{-Mod}$ be KLR algebra modules. Then 
\begin{equation*}
\begin{split}
\textrm{Res}(M \circ_B N)&=\textrm{Res}(M)\circ_A \textrm{Res}(N) \\
\textrm{Ind}(P \circ_A Q)&=\textrm{Ind}(P)\circ_B \textrm{Ind}(Q).
\end{split}
\end{equation*}
\end{lemma}
Let $\vvn:=\bigoplus_{\nu} \vv$ where the direct sum runs over all $\nu \in \thetani$ with $|\nu|=2n$. Then set $\w:=\bigoplus_{n \in \mathbb{Z}_{\geq 0}} \vvn$. Similarly, put $\textbf{R}:=\bigoplus_{m \in \mathbb{Z}_{\geq 0}} \textbf{R}(m)$ and $\textbf{R}(m):=\bigoplus_{\tilde{\nu} } \subklr$ where the second direct sum runs over all $\tilde{\nu} \in \mathbb{N}I^+$ such that $|\tilde{\nu}|=m$.  
\begin{corollary}
There is a $\mathbb{Z}[v,v^{-1}]$-algebra isomorphism of graded algebras 
\begin{equation*}
K_0(\w\textrm{-gmod})\cong K_0(\textbf{R}\textrm{-gmod}).
\end{equation*}
\end{corollary}
For $\alpha\in \thetani$, $|\alpha|=2m$ we put $\textbf{e}=\sum_{\textbf{i} \in I^{\alpha^+}}\idemp$. Then $\textbf{e}\vva$ is free as a left $\textbf{R}_{\alpha^+}$-module on a basis given by minimal length right coset representatives of $\sm$ in $W^B_m$. It follows that $\Res:\vva\textrm{-mod} \longrightarrow \textbf{R}_{\alpha^+}\textrm{-mod}$; $M \mapsto \textbf{e}M$ sends projective modules to projective modules. We also know that $\Ind: \textbf{R}_{\alpha^+}\textrm{-mod} \longrightarrow \vva\textrm{-mod}$; $N \mapsto \vva\textbf{e}\otimes_{\textbf{R}_{\alpha^+}}N$ sends projective modules to projective modules. These remarks, together with the fact that $\textbf{R}\textrm{-proj}$ is stable under the convolution product, show that $\w\textrm{-proj}$ is also stable under the convolution product and hence $K_0(\w\textrm{-proj})$ is a $\mathbb{Z}[v,v^{-1}]$-algebra.
\\
\\
Suppose $|\beta|=2m$, $|\gamma|=2n$. Let $\w_{\beta^+,\gamma^+}$ denote the image of $\chi$. We define a map $R(M)\otimes R(N) \longrightarrow N\circ M$ by $p\otimes q \mapsto \varphi_{w[n,m]}(q\otimes p)$, where $\varphi_{w[n,m]}$ is the element given in (\ref{equation:the element w[m,n]}). By Lemma \ref{lemma:varphi satisfies the braid rels} this map is $\w_{\beta^+,\gamma^+}$-linear. Then we can extend it to a $\vvbg$-module morphism $M\circ N \longrightarrow N\circ M$. These maps satisfy the commutative diagrams given in the previous subsection and satisfy the Yang-Baxter equation.
\begin{lemma}
Take $\alpha,\beta \in \thetani$, $L \in \vva\textrm{-Mod}$ and $M \in \vvb\textrm{-Mod}$. Then 
\begin{equation*}
R_{L,M}(L\circ_B M)=I(R_{L,M}(R(L)\circ_A R(M))).
\end{equation*}
\end{lemma}
\begin{proof}
On the one hand, 
\begin{equation*}
\begin{split}
R_{L,M}(L\circ_B M)&=R_{L,M}(\vvab e(\alpha,\beta)\otimes_{\textbf{R}_\alpha\otimes \textbf{R}_\beta} R(L)\otimes R(M)) \\
&=\vvba e(\alpha, \beta)\varphi_{w[m,l]} \otimes_{\textbf{R}_\beta\otimes \textbf{R}_\alpha} R(M)\otimes R(L)
\end{split}
\end{equation*}
whilst on the other hand
\begin{equation*}
\begin{split}
I(R_{L,M}(R(L)\circ_A R(M)))&=I(R_{L,M}(\textbf{R}_{\alpha+\beta}e(\alpha,\beta)\otimes_{\textbf{R}_\alpha \otimes \textbf{R}_\beta} R(L)\otimes R(M))) \\
&=\vvba e\otimes_{e\w e} \textbf{R}_{\beta+\alpha}e(\alpha,\beta)\varphi_{w[m,l]} \otimes_{\textbf{R}_\beta \otimes \textbf{R}_\alpha} (R(M)\otimes R(L)) \\
&\cong \vvba e(\alpha, \beta)\varphi_{w[m,l]} \otimes_{\textbf{R}_\beta \otimes \textbf{R}_\alpha} (R(M)\otimes R(L)). \\
\end{split}
\end{equation*}
\end{proof}
In the following example we compare this convolution product to our naive construction and find they are compatible. That is, the resulting products are isomorphic. Take $\beta,\gamma\in \thetani$ such that $\supp(\beta)\cap\supp(\gamma)=\emptyset$, $\supp(\beta)\nleftrightarrow \supp(\gamma)$ and compute $M\circ N$, as well as the corresponding $R$-matrix, using the restriction-induction method. We first do this using the data from Example \ref{example:vv algebra r-matrix 1}. \\
\\
Let us again write $I$ and $R$ for the induction and restriction functors respectively.
\begin{example}\label{example:VV example}
Let $\beta=\lambda+\lambda^{-1}$, $\gamma=p^4\lambda+p^{-4}\lambda^{-1}$ and consider the corresponding VV algebras. Recall the notation from \ref{notation}. Let 
\begin{equation*}
M=\w_\beta e(0)/ \textrm{rad}(\w_\beta e(0)) \in \w_\beta\textrm{-mod}. 
\end{equation*}
This is a simple module. As a k-vector space it has basis $\{ e(0),\pi e(0) \}$. Then let $R(M):=\textrm{Res}(M) \in \textbf{R}_{\alpha^+}\textrm{-mod}$, which has basis $\{ e(0) \}$. Similarly, let 
\begin{equation*}
N=\w_\gamma e(4) / \textrm{rad}(\w_\gamma e(4)) \in \w_\gamma\textrm{-mod.} 
\end{equation*}
Then $N$ is also a simple module and has k-basis $\{ e(4),\pi e(4) \}$. Then let $R(N):=\textrm{Res}(N) \in \textbf{R}_{\beta^+}$, which has basis $\{ e(4) \}$.
\\
\\
Now we calculate the convolution product of the KLR algebra modules;
\begin{equation*}
R(M) \circ R(N)=\textbf{R}_{\alpha^++\beta^+}e(\alpha^+\beta^+) \underset{\textbf{R}_{\alpha^+}\otimes \textbf{R}_{\beta^+}}{\otimes} (R(M) \otimes R(N)).
\end{equation*}
It has a basis $\{ e(e(0)\otimes e(4)), \sigma_1(e(0)\otimes e(4)) \}$.
\\
\\
Then, by after applying the induction functor, the convolution product of the VV algebra modules $M \circ N$ has basis
\begin{equation*}
\left\lbrace
\begin{tabular}{llll}
$e(e(0)\otimes e(4))$, & $\pi (e(0)\otimes e(4))$, &
$\sigma_1\pi\sigma_1(e(0)\otimes e(4))$, & $\pi\sigma_1\pi\sigma_1(e(0)\otimes e(4))$, \\
$\sigma_1(e(0)\otimes e(4))$, & $\sigma_1\pi(e(0)\otimes e(4))$, &
 $\pi\sigma_1(e(0)\otimes e(4))$, & $\pi\sigma_1\pi(e(0)\otimes e(4))$
\end{tabular}
\right\rbrace.
\end{equation*}
Similarly, we find that $N\circ M$ has basis
\begin{equation*}
\left\lbrace
\begin{tabular}{llll}
$e(e(4)\otimes e(0))$, & $\pi (e(4)\otimes e(0))$, &
$\sigma_1\pi\sigma_1(e(4)\otimes e(0))$, & $\pi\sigma_1\pi\sigma_1(e(4)\otimes e(0))$, \\
$\sigma_1(e(4)\otimes e(0))$, & $\sigma_1\pi(e(4)\otimes e(0))$, &
$\pi\sigma_1(e(4)\otimes e(0))$, & $\pi\sigma_1\pi(e(4)\otimes e(0))$
\end{tabular}
\right\rbrace.
\end{equation*}
What is the corresponding $R$-matrix $R_{M,N}:M\circ N \longrightarrow N\circ M$ with respect to this construction? Again, $\varphi_{w[1,1]}=\varphi_1=\sigma_1$. After ordering the basis in a suitable way we find
\begin{equation*} R_{M,N}=
\setcounter{MaxMatrixCols}{20}
\begin{pmatrix} 
0&1&0&0&0&0&0&0 \\ 
1&0&0&0&0&0&0&0 \\ 
0&0&0&0&0&1&0&0 \\ 
0&0&0&0&1&0&0&0 \\ 
0&0&0&1&0&0&0&0 \\ 
0&0&1&0&0&0&0&0 \\ 
0&0&0&0&0&0&0&1 \\ 
0&0&0&0&0&0&1&0
\end{pmatrix}.
\end{equation*}
We find that this coincides with our so-called naive construction, see Example \ref{example:vv algebra r-matrix 1}.
\end{example}
Let us now compute examples in which there exist arrows between $\supp(\beta)$ and $\supp(\gamma)$ in $\Gamma$.
\begin{example} \label{example:arrows between supports}
Let $\beta=\lambda+\lambda^{-1}+p^2\lambda+p^{-2}\lambda^{-1}$ and let $\gamma=\lambda+\lambda^{-1}$. Let 
\begin{equation*}
M=\w_\beta e(02)/ \textrm{rad}(\w_\beta e(02)) \in \w_\beta\textrm{-mod}
\end{equation*}
which is 4 dimensional with basis $\{ e(02),\pi e(02), \sigma_1\pi e(02),\pi\sigma_1\pi e(02) \}$. Then $R(M)$ is 1 dimensional with basis $\{ e(02) \}$.
Let 
\begin{equation*}
N=\w_\gamma e(0)/ \textrm{rad}(\w_\gamma e(0)) \in \w_\gamma\textrm{-mod}. 
\end{equation*}
Then $N$ has basis $\{ e(0),\pi e(0) \}$ and $R(N)$ has basis $\{ e(0) \}$. Now we can easily calculate a basis of $M\circ N$, $N\circ M$ using Lemma \ref{lemma:the form of the conv prod}. Let us order the elements of $\mathcal{D}(W^B_3/(\mathfrak{S}_2\times \mathfrak{S}_1))$ by length. Then $M\circ N$ has basis
\begin{equation*}
\left\lbrace
\begin{array}{lll}
e(e(02)\otimes e(0)) & \sigma_1\sigma_2 (e(02)\otimes e(0)) &  \\
\pi (e(02)\otimes e(0)) & \pi\sigma_2 (e(02)\otimes e(0)) & \ldots\ldots \\
\sigma_2 (e(02)\otimes e(0)) & \pi\sigma_1\pi (e(02)\otimes e(0)) &   \\
\sigma_1\pi (e(02)\otimes e(0)) & \sigma_2\sigma_1\pi (e(02)\otimes e(0)) & 
\end{array}
\right\rbrace
\end{equation*}
and is 24 dimensional in total. We calculate the image of $M\circ N$ under $R_{M,N}$. Note that $\varphi_{w[n,m]}=\varphi_2\varphi_1$.
\begin{equation*}
\begin{split}
R_{M,N}(e(e(02)\otimes e(0)))&=\varphi_2\varphi_1(e(0)\otimes e(02)) \\
&=\sigma_2(\sigma_1x_1-x_1\sigma_1)(e(0)\otimes e(02))\\
&=-\sigma_2x_1\sigma_1(e(0)\otimes e(02))\\
&=-\sigma_2(\sigma_1x_2-e)(e(0)\otimes e(02))\\
&=\sigma_2(e(0)\otimes e(02))\\
&=0.
\end{split}
\end{equation*}
But then, for any other basis element $\sigma_\eta(e(02)\otimes e(0))$ of $M\circ N$, we have
\begin{equation*}
R_{M,N}(\sigma_\eta(e(02)\otimes e(0)))=\sigma_\eta\varphi_{w[n,m]}(e(0)\otimes e(02))=0.
\end{equation*}
Then $R_{M,N}=0$ here. We resolve this problem by introducing spectral parameters and the renormalised $R$-matrix $r_{M,N}$.
\end{example}
\subsection{Spectral Parameters}
For an indeterminate $z$, homogeneous of degree 2, we recall the algebra morphism $\psi_z$ from \cite{KKKI}. Take $\alpha \in \mathbb{N}I$. Then $\psi_z$ is the morphism
\begin{equation*}
\begin{split}
\psi_z:\klra &\longrightarrow k[z] \otimes \klra \\
e(i) &\mapsto 1\otimes e(i) \\
\sigma_k &\mapsto 1\otimes \sigma_k \\
x_l &\mapsto z\otimes 1+1\otimes x_l.
\end{split}
\end{equation*}
Given $L \in \klra\textrm{-Mod}$, we define $L_z:=k[z]\otimes L \in (k[z] \otimes \klra)\textrm{-Mod}$, where the action of $1\otimes x$, for $x \in \klra$, is given by multiplication by $\psi_z(x)$. We may also view $L_z$ as a $\klra$-module. Let $z^\prime$ be another indeterminate of degree 2, with $z^\prime \neq z$, and take $L^\prime \in \klrg\textrm{-Mod}$, for some $\gamma \in \mathbb{N}I$. Then we can calculate $L_z \circ L^\prime_{z^\prime}$ as modules with this twisted action.
\\
\\
For $M \in \vvb\textrm{-mod}$, we define 
\begin{equation*}
M_z:=(k[z]\otimes \vvb \textbf{e})\otimes_{k[z]\otimes \textbf{e}\vvb\textbf{e}} R(M)_z \in (k[z]\otimes \vvb) \textrm{-Mod}
\end{equation*}
where again $R(M):= \textbf{e}M$, $\textbf{e}=\sum_{\textbf{i} \in I^{\beta^+}} \idemp$ and $R(M)_z \in \big(k[z] \otimes \klrb\big) \textrm{-Mod}$. We also remark that $M_z$ may be considered as a $\vvb$-module.
\\
\\
Considering $M_z$ as a $\vvb$-module, for $M \in \vvb\textrm{-mod}$, note that we have $R(M_z)=R(M)_z$;
\begin{equation*}
R(M_z)=\big(k[z]\otimes \textbf{e}\vvb \textbf{e}\big) \otimes_{k[z] \otimes \textbf{e}\vvb \textbf{e}} R(M)_z \cong R(M)_z. 
\end{equation*}
Then,
\begin{equation}\label{equation:spectral parameters}
M_z\circ N_{z^\prime}=k[z,z^\prime]\otimes\w_{\beta+\gamma}e(\beta^+\gamma^+)\otimes_{k[z]\otimes \textbf{R}_{\beta^+}\otimes k[z^\prime]\otimes \textbf{R}_{\gamma^+}} \big(k[z] \otimes R(M) \otimes k[z^\prime] \otimes R(N)\big)
\end{equation}
The \textbf{renormalized $R$-matrix} $r_{M,N}:M\circ N \longrightarrow N\circ M$ is defined as
\begin{equation*}
r_{M,N}:=((z^\prime-z)^{-s}R_{M_z,N_{z^\prime}})|_{z,z^\prime =0}
\end{equation*}
where $s$ is the largest non-negative integer such that $(z^\prime -z)^s$ is a factor of the image of $R_{M_z,N_{z^\prime}}$. We can make this morphism homogeneous by shifting the grading by $-(\beta,\gamma)+2[\beta,\gamma]+2s$.
\\
\\
Example \ref{example:arrows between supports} continued: Let us continue with this example using spectral parameters. Consider the basis element $e(1\otimes e(02) \otimes 1 \otimes e(0)) \in M_z\circ N_{z^\prime}$.
\begin{equation*}
R_{M_z,N_{z^\prime}}(e(1\otimes e(02) \otimes 1 \otimes e(0))) \in N_{z^\prime}\circ M_z=(z^\prime-z)\sigma_2\sigma_1(1\otimes e(0)\otimes 1 \otimes e(02))
\end{equation*}
which is now non-zero. One can do similar calculations for other basis elements.
\begin{example}\label{example:VV example II}
Let us now consider a smaller example; $\beta=\lambda+\lambda^{-1}$ and $\gamma=p^2\lambda+p^{-2}\lambda^{-1}$. Let us again take 
\begin{equation*}
M=\w_\beta e(0)/ \textrm{rad}(\w_\beta e(0)) \in \w_\beta\textrm{-mod}
\end{equation*}
which is 2 dimensional with basis $\{ e(0),\pi e(0)\}$. Then $R(M)$ is 1 dimensional with basis $\{ e(0) \}$.
Let 
\begin{equation*}
N=\w_\gamma e(2)/ \textrm{rad}(\w_\gamma e(2)) \in \w_\gamma\textrm{-mod} 
\end{equation*}
which is 2 dimensional with basis $\{ e(2),\pi e(2)\}$. Then $R(N))$ is 1 dimensional with basis $\{ e(2) \}$. Then
\begin{equation*}
\begin{split}
M\circ N&= \big\langle \sigma_w\big( e(0)\otimes e(2) \big) \mid w \in \mathcal{D}  \big\rangle \\
N\circ M&= \big\langle \sigma_w\big( e(2)\otimes e(0) \big) \mid w \in \mathcal{D}  \big\rangle 
\end{split}
\end{equation*}
where $\mathcal{D}:=\mathcal{D}(W^B_2/(\mathfrak{S}_1\times \mathfrak{S}_1))=W^B_2$. One finds that $R_{M,N}$ is given by
\begin{equation*} R_{M,N}=
\setcounter{MaxMatrixCols}{20}
\begin{pmatrix} 
0&0&0&0&0&0&0&0 \\ 
0&0&0&0&0&0&0&0 \\ 
1&0&0&0&0&0&0&0 \\ 
0&0&0&0&0&0&0&0 \\ 
0&1&0&0&0&0&0&0 \\ 
0&0&0&1&0&0&0&0 \\ 
0&0&0&0&0&0&0&0 \\ 
0&0&0&0&0&0&1&0
\end{pmatrix}
\end{equation*}
which is clearly not invertible. One can also check that $R_{N,M}=R_{M,N}$.
\\
\\
We do the same but now with spectral parameters, using \ref{equation:spectral parameters}. One finds that the image of basis elements which were mapped to something non-zero by $R_{M,N}$ remain largely unchanged after introducing spectral parameters. So let us focus on the basis elements that are mapped to zero by $R_{M,N}$. We have,
\begin{equation*}
\begin{split}
R_{M_z,N_{z^\prime}}(\sigma_1(1\otimes e(0)\otimes 1 \otimes e(2)))&=(z^\prime-z)(1\otimes e(2)\otimes 1 \otimes e(0))\\
R_{M_z,N_{z^\prime}}(\pi\sigma_1(1\otimes e(0)\otimes 1 \otimes e(2)))&=(z^\prime-z)\pi(1\otimes e(2)\otimes 1 \otimes e(0))\\
R_{M_z,N_{z^\prime}}(\sigma_1\pi\sigma_1(1\otimes e(0)\otimes 1 \otimes e(2)))&=(z^\prime-z)\sigma_1\pi(1\otimes e(2)\otimes 1 \otimes e(0))\\
R_{M_z,N_{z^\prime}}(\pi\sigma_1\pi\sigma_1(1\otimes e(0)\otimes 1 \otimes e(2)))&=(z^\prime-z)\pi\sigma_1\pi(1\otimes e(2)\otimes 1 \otimes e(0)).
\end{split}
\end{equation*}
In other words, $R_{M_z,N_{z^\prime}}$ is invertible. Recall the renormalised $R$-matrix $r_{M,N}:=((z^\prime -z)^{-s}R_{M_z,N_{z^\prime}})|_{z,z^\prime=0}$, where $s$ is the largest non-negative integer such that $(z^\prime - z)^s$ is a factor of $R_{M_z,N_{z^\prime}}(M_z\circ N_{z^\prime})$. Since $s=0$ here, we find that $r_{M,N}=R_{M,N}$ in this example. Similarly, $r_{N,M}=R_{N,M}$.
\\
\\
Consider the short exact sequence
\begin{equation*}
0 \longrightarrow \textrm{ker}(r_{M,N}) \longrightarrow M\circ N \longrightarrow \textrm{im}(r_{M,N}) \longrightarrow 0.
\end{equation*}
One finds that, as a k-vector space, 
\begin{equation*}
\textrm{ker}(r_{M,N})=\big\langle \sigma_1(e(0)\otimes e(2)),\pi\sigma_1(e(0)\otimes e(2)), \sigma_1\pi\sigma_1(e(0)\otimes e(2)),\pi\sigma_1\pi\sigma_1(e(0)\otimes e(2)) \big\rangle 
\end{equation*}
and is a simple $\w_{\beta+\gamma}$-module. We also have
\begin{equation*}
\textrm{im}(r_{M,N})=\big\langle \sigma_1(e(2)\otimes e(0)),\pi\sigma_1(e(2)\otimes e(0)), \sigma_1\pi\sigma_1(e(2)\otimes e(0)),\pi\sigma_1\pi\sigma_1(e(2)\otimes e(0)) \big\rangle 
\end{equation*}
which is also a simple $\w_{\gamma+\beta}$-module. Then, in the Grothendieck group, one has $[M\circ N]=[\textrm{ker}(r_{M,N})]+[\textrm{im}(r_{M,N})]$. In the Grothendieck ring this is $[M][N]=[\textrm{ker}(r_{M,N})]+[\textrm{im}(r_{M,N})]$.
\\
\\
Let us look at the structure of $M\circ N$ (the structure of $N\circ M$ will be similar). It has two simple subquotients, $L(1)$ and $L(2)$;
\begin{center}
\begin{tikzpicture}
\node(A) at (-0.5,0) {$M\circ N:$};
\node[anchor=west] (B) at (1,1) {$L(1) =\langle e(0\otimes 2),\pi(0\otimes 2),\sigma_1\pi(0\otimes 2),\pi\sigma_1\pi(0\otimes 2) \rangle$};
\node[anchor=west] (C) at (1,-1) {$L(2) =\langle \sigma_1e(0\otimes 2),\pi\sigma_1(0\otimes 2),\sigma_1\pi\sigma_1(0\otimes 2),\pi\sigma_1\pi\sigma_1(0\otimes 2) \rangle.$};
\draw (1.4,0.70) -- (1.4,-0.70);
\end{tikzpicture}
\end{center}
So we find here that $M\circ N$, $N\circ M$ both have a simple socle and a simple head and $\textrm{ker}(r_{M,N})=\textrm{soc}(M\circ N)$, $\textrm{im}(r_{M,N})=\textrm{soc}(N\circ M)$. Then we have $\textrm{soc}(N\circ M)\cong\textrm{hd}(M\circ N)$ because
\begin{equation*}
\textrm{im}(r_{M,N})=\textrm{soc}(N\circ M)\cong (M\circ N)/\textrm{ker}(r_{M,N})\cong \textrm{hd}(M\circ N).
\end{equation*}
We now calculate $R_{M,M}$ where $M$ is the same module as above. We have $M\circ M = \big\langle \sigma_w(e(0)\otimes e(0)) \mid w \in W^B_2 \big\rangle$. One can easily check that $M\circ M$ is a simple module. Hence $M$ is real and we can check that $N$ is also real.
\\
\\
Since
\begin{equation*}
\begin{split}
e(e(0)\otimes e(0)) \mapsto \varphi_1(e(0)\otimes e(0)) &= (\sigma_1x_1-x_1\sigma_1)(e(0)\otimes e(0)) \\
&=e(e(0)\otimes e(0))
\end{split}
\end{equation*}
we find that $R_{M,M}=r_{M,M}=\textrm{Id}_{M\circ M}$. In fact the following results show that the properties in this example can be generalised.
\end{example}
The following two results are analogues of Lemma 3.1 and Theorem 3.2 from \cite{SimplicityOfHeadsSocles}.
\begin{lemma}\label{lemma:submodule lemma}
Suppose $\beta_k \in \thetani$ and $M_k \in \w_{\beta_k}\textrm{-mod}$, where $k=1,2,3$. Suppose further that
\begin{itemize}
\item[$\bigcdot$] $X\subset M_1\circ M_2$ is a $\w_{\beta_1+\beta_2}$-submodule.
\item[$\bigcdot$] $Y\subset M_2\circ M_3$ is a $\w_{\beta_2+\beta_3}$-submodule.
\end{itemize}
such that $X\circ M_3 \subset M_1\circ Y$ as submodules of $M_1\circ M_2\circ M_3$. Then there exists an $\w_{\beta_2}$-submodule $N\subset M_2$ with $X\subset M_1\circ N$ and $N\circ M_3\subset Y$.
\end{lemma}
\begin{proof}
The proof is similar to the proof of the corresponding result in \cite{KKKO}. We need only note that $\mathcal{D}(W^B_{n_1+n_2}/\mathfrak{S}_{n_1}\times \mathfrak{S}_{n_2}) \subset \mathcal{D}(W^B_{n_1+n_2+n_3}/\mathfrak{S}_{n_1}\times \mathfrak{S}_{n_2+n_3})$, where $2n_k=\textrm{ht}(\beta_k)$.
\end{proof}
We now have the following theorem for VV algebras which is a replica of \cite{SimplicityOfHeadsSocles}, Theorem 3.2.
\begin{theorem}\label{theorem:analogue of KKKO}
Suppose $\beta,\gamma\in \thetani$. Take a non-zero module $M\in \w_\beta\textrm{-mod}$ such that $r_{M,M}\in k\textrm{id}_{M\circ M}$ and a simple module $N\in \w_\gamma\textrm{-mod}$.  Then,
\begin{itemize}
\item[(i)] $M\circ N$ and $N\circ M$ both have simple socles and simple heads. 
\item[(ii)] $\textrm{im}(r_{M,N})=\textrm{soc}(N\circ M)=\textrm{hd}(M\circ N)$ and \\ $\textrm{im}(r_{N,M})=\textrm{soc}(M\circ N)=\textrm{hd}(N\circ M)$.
\item[(iii)] $M$ is a simple module.
\end{itemize}
\end{theorem}
\begin{proof}
For the convenience of the reader we recall \cite{SimplicityOfHeadsSocles}, Theorem 3.2. For half of $(i)$ and $(ii)$ we will prove that $M\circ N$ has a simple socle and that this unique simple submodule is $\textrm{im}(r_{N,M})$. Let $S \subset M\circ N$ be any non-zero $\w_{\beta+\gamma}$-submodule. Let $m,m^\prime$ be the order of zeroes of $R_{N,M_z}, R_{M,M_z}$ respectively. By definition this means,
\begin{equation*}
\begin{split}
r_{N,M}=(z^{-m}R_{N,M_z})|_{z=0}:N\circ M&\longrightarrow M\circ N \\
r_{M,M}=(z^{-m^\prime}R_{M,M_z})|_{z=0}:M\circ M&\longrightarrow M\circ M.
\end{split}
\end{equation*}
We have the following commutative diagram
\begin{center}
\begin{tikzpicture}
\matrix (m) [matrix of math nodes, row sep=4em,
column sep=7em, text height=1.5ex, text depth=0.25ex]
{S\circ M_z &  & M_z\circ S\\
M\circ N\circ M_z & M\circ M_z \circ N & M_z\circ M\circ N. \\};
\path[->]
(m-1-1) edge node [above] {$z^{-m-m^\prime} R_{S,M_z}$} (m-1-3)
(m-2-1) edge node [above] {$M\circ z^{-m}R_{N,M_z}$} (m-2-2)
(m-2-2) edge node [above] {$z^{-m^\prime}R_{M,M_z}\circ N$} (m-2-3);
\path[>=stealth,right hook->]
(m-1-3) edge (m-2-3)
(m-1-1) edge (m-2-1);
\end{tikzpicture}
\end{center}
When we set $z=0$ we obtain
\begin{center}
\begin{tikzpicture}
\matrix (m) [matrix of math nodes, row sep=4em,
column sep=7em, text height=1.5ex, text depth=0.25ex]
{S\circ M &  & M\circ S\\
M\circ N\circ M & M\circ M \circ N & M\circ M\circ N \\};
\path[->]
(m-1-1) edge (m-1-3)
(m-2-1) edge node [above] {$M\circ r_{N,M}$} (m-2-2)
(m-2-2) edge node [above] {$k\textrm{id}_{M\circ M\circ N}$} (m-2-3);
\path[>=stealth,right hook->]
(m-1-3) edge (m-2-3)
(m-1-1) edge (m-2-1);
\end{tikzpicture}
\end{center}
because of the assumption $r_{M,M}=k\textrm{id}_{M\circ M}$. Then we have $(M\circ r_{N,M})(S\circ M) \subset M\circ S$, i.e.\ $S\circ M \subset M\circ r^{-1}_{N,M}(S)$. From Lemma \ref{lemma:submodule lemma} we know that there exists $K\subset N$ such that $S\subset M\circ K$ and $K\circ M \subset r^{-1}_{N,M}(S)$. It follows that $K\neq 0$ and, since $N$ is simple, $N=K$. So $N\circ M \subset r_{N,M}^{-1}(S)$ and $\textrm{im}(r_{N,M})\subset S$. But $S$ was any non-zero submodule of $M\circ N$ and so $\textrm{im}(r_{M,N})$ is the unique non-zero simple submodule of $M\circ N$. 
\\
\\
For $(iii)$ we take $N=k \in \w_0\textrm{-mod}$, where $k$ is the ground field. Then $M\circ k \cong M \cong k\circ M$ and $(iii)$ follows from $(i)$ and $(ii)$.
\end{proof}
\begin{corollary} \label{corollary:TFAE real simple, identity map}
Take $0 \neq M \in \vvb\textrm{-mod}$. Then the following are equivalent;
\begin{itemize}
\item[(i)] $M$ is a real simple $\vvb$-module.
\item[(ii)] $r_{M,M} \in k\hspace{0.2em}\textrm{id}$.
\item[(iii)] $\End_{\w_{2\beta}}(M\circ M)=k\hspace{0.2em}\textrm{id}_{M\circ M}$.
\end{itemize}
\end{corollary}
\begin{proof}
If $M$ is a real simple module then, by definition, $M\circ M$ is simple and so $(i)$ implies $(iii)$. It is immediate that $(iii)$ implies $(ii)$. Finally, $(ii)$ implies $(i)$ using Theorem \ref{theorem:analogue of KKKO}.
\end{proof}
\begin{corollary}
For $\beta \in \thetani$, let $M$ be a real simple $\vvb$-module. Then, for $n \in \mathbb{N}$, $M^{\circ n}$ is a real simple $\w_{n\beta}$-module.
\end{corollary}
\begin{proof}
We have the following commutative diagram
\begin{center}
\begin{tikzpicture}
\matrix (m) [matrix of math nodes, row sep=4em,
column sep=5em, text height=1.5ex, text depth=0.25ex]
{M\circ M\circ M & M\circ M\circ M \\
  & M\circ M\circ M \\};
\path[-stealth]
(m-1-1) edge node [above] {$r_{M,M}$} (m-1-2)
        edge node [left] {$r_{M,M\circ M}$} (m-2-2)
(m-1-2) edge node [right] {$r_{M,M}$} (m-2-2);
\end{tikzpicture}
\end{center}
and, using Corollary \ref{corollary:TFAE real simple, identity map}, we have that $r_{M,M} \in k\hspace{0.2em}\textrm{id}_{M,M}$. Therefore $r_{M,M\circ M} \in k\hspace{0.2em}\textrm{id}_{M^{\circ 3}}$ and similarly $r_{M\circ M,M} \in k\hspace{0.2em}\textrm{id}_{M^{\circ 3}}$. We now repeat this to obtain $r_{M\circ (M\circ M)}\circ r_{M\circ (M\circ M)}=r_{M\circ M,M\circ M} \in k\hspace{0.2em}\textrm{id}_{M^{\circ 4}}$. By Corollary \ref{corollary:TFAE real simple, identity map} this implies that $M\circ M$ is a real simple module. Continuing inductively yields the result.
\end{proof}
\section{Quantum Cluster Algebra Structure on $\btgh$}\label{section:conjectures}
Recall that we fix a quiver $\Gamma$ with vertex set $I$ as described in \ref{section:vv definition}. To emphasise, we repeat that we do this in such a way that $\pm q, \pm p \not\in I$. Consider the $k$-algebra
\begin{equation*}
\vvn:=\bigoplus_{\nu} \vv
\end{equation*}
where the direct sum runs over all $\nu \in \thetani$ with $|\nu|=2n$. Then set
\begin{equation*}
\w_I:=\bigoplus_{n \in \mathbb{Z}_{\geq 0}} \vvn.
\end{equation*}
We will often write $\w$ instead of $\w_I$ keeping in mind that we have already fixed $I$ in the way described above. Let $v$ be an indeterminate and write $\mathcal{A}:=\mathbb{Z}[v,v^{-1}]$. Let $K_0(\w(n)\textrm{-proj})$ denote the Grothendieck group of the category of finitely generated graded projective $\w(n)$-modules and set
\begin{equation*}
K_0(\w\textrm{-proj}):=\bigoplus_{n \in \mathbb{Z}_{\geq 0}} K_0(\w(n)\textrm{-proj}). 
\end{equation*}
We recall that $K_0(\w\textrm{-proj})$ is an $\mathcal{A}$-module, where $v$ and $v^{-1}$ act by shifting the grading by $1$ and $-1$, respectively.
\\
\\
We now recall the algebra $\btg$, defined by Enomoto and Kashiwara, following \cite{EnomotoKashiwara}. Put $K:=\mathbb{Q}(v)$. We consider $I$ to be an indexing set for the set of simple roots of the associated Lie algebra $\lie$ and we let $Q$ be the free $\mathbb{Z}$-module with basis $\{ \alpha_i : i \in I\}$. We take a symmetric bilinear form $(\cdot \hspace*{0.2em},\cdot):Q\times Q \longrightarrow \mathbb{Z}$ such that $(\alpha_i,\alpha_i)/2 \in \mathbb{Z}_{>0}$ for all $i$ and such that $(\alpha_i^\vee,\alpha_j)\in \mathbb{Z}_{\leq 0}$ when $i \neq j$. As usual we have $\alpha_i^\vee :=2\alpha_i/(\alpha_i,\alpha_i)$.
\begin{defn}
The algebra $\btg$ is the $K$-algebra generated by elements $E_i$, $F_i$ and invertible elements $T_i$, for $i \in I$, which satisfy the following relations.
\begin{enumerate}
\item $T_iT_j=T_jT_i$ for all $i,j \in I$,
\item $T_{\theta(i)}=T_i$ for all $i \in I$,
\item $T_iE_jT_i^{-1}=v^{(\alpha_i+\alpha_{\theta(i)},\alpha_j)}E_j$ and $T_iF_jT_i^{-1}=v^{(\alpha_i+\alpha_{\theta(i)},-\alpha_j)}F_j$ for all $i,j \in I$,
\item $E_iF_j=v^{-(\alpha_i,\alpha_j)}F_jE_i+\delta_{i,j}+\delta_{\theta(i),j}T_i$ for all $i,j \in I$,
\item the $E_i$ and the $F_i$ satisfy the Serre relations; for all $i\neq j \in I$,
\begin{equation*}
\sum_{k=0}^b(-1)^kE_i^{(k)}E_jE_i^{(b-k)}=0, \quad \sum_{k=0}^b(-1)^kF_i^{(k)}F_jF_i^{(b-k)}=0,
\end{equation*}
where $b=1-(\alpha_i^\vee,\alpha_j)$ and
\begin{equation*}
E_i^{(k)}=E_i^k/[k]_i!, \hspace*{0.3em} F_i^{(k)}=F_i^k/[k]_i!, \hspace*{0.3em}
[k]_i=(v_i^k-v_i^{-k})/(v_i-v_i^{-1}), \hspace*{0.3em}[k]_i!=[1]_i\cdots [k]_i.
\end{equation*}
\end{enumerate}
\end{defn}
Note that $\btg\cong U_v^-(\lie) \otimes K[T_i^{\pm 1}; i \in I] \otimes U_v^+(\lie)$. Now let $\lambda \in P_+$, where
\begin{equation*}
P_+:=\{ \lambda \in \Hom(Q, \mathbb{Q}) \mid \langle \alpha_i^\vee,\lambda \rangle \in \mathbb{Z}_{\geq 0} \textrm{ for any } i \in I \},
\end{equation*}
be a dominant integral weight such that $\theta(\lambda)=\lambda$. Enomoto and Kashiwara showed that $\btg$ has a simple highest weight module $\vtl$, with highest weight vector $\phi_\lambda$ associated to the highest weight $\lambda$, which is unique up to isomorphism. That is, $E_i\phi_\lambda=0$, for all $i \in I$, and $T_i\phi_\lambda=v^{(\alpha_i,\lambda)}\phi_\lambda$, for all $i \in I$.
\\
\\
In our fixed setting ($\pm p, \pm q \not\in I$) note that we have either $\lie=\glinfty \oplus \glinfty$ or $\lie=A^{(1)}_r \oplus A^{(1)}_r$, depending on whether or not $p$ is a primitive root of unity. Furthermore, we consider the case $\lambda=0$ and we have
\begin{equation*}
(\alpha_i,\alpha_j)=\left\lbrace
\begin{array}{ll}
-|i\rightarrow j|-|j\rightarrow i| & \textrm{ if }i\neq j \\
2 & \textrm{ if }i=j
\end{array}
\right.
\end{equation*}
where, as before, $|i\rightarrow j|$ denotes the number of arrows in $\Gamma$ which have origin $i$ and target $j$.
\begin{remark}
Since $\vtz$ is a simple $\btg$-module there is an isomorphism of left modules
\begin{equation*}
\vtz \cong \btg/\mathfrak{m}
\end{equation*}
where $\mathfrak{m}$ is a maximal left ideal in $\btg$. Let us denote $\btgh:= \btg/\mathfrak{m}$.
\end{remark}
Varagnolo and Vasserot \cite{VaragnoloVasserot} proved, in full generality, that $K\otimes_{\mathcal{A}} K_0(\w\textrm{-proj})$ is a left $\btg$-module and that there is a $\btg$-module isomorphism 
\begin{equation*}
K\otimes_{\mathcal{A}} K_0(\w\textrm{-proj}) \cong \vtl.
\end{equation*}
Furthermore, $K\otimes_{\mathcal{A}} K_0(\w\textrm{-proj})$ has an $\mathcal{A}$-basis in which basis elements correspond to isomorphism classes of certain simple graded $\w$-modules. Since there exists Morita equivalence between VV algebras and KLR algebras in our fixed setting it follows that the basis elements are in correspondence with certain simple graded $\textbf{R}$-modules, where
\begin{equation*}
\textbf{R}:=\bigoplus_{m \in \mathbb{Z}_{\geq 0}} \textbf{R}(m), \quad \quad \quad \textbf{R}(m):=\bigoplus_{\tilde{\nu} } \subklr.
\end{equation*}
The second direct sum runs over all $\tilde{\nu} \in \mathbb{N}I^+$ such that $|\tilde{\nu}|=m$. We have shown that $\w\textrm{-proj}$ is closed under our newly-defined convolution product so that $K\otimes_{\mathcal{A}} K_0(\w\textrm{-proj})$, and hence $\btgh$, is a $\mathbb{Q}(v)$-algebra. We conjecture that $K\otimes_{\mathcal{A}}K_0(\w\textrm{-proj})$ is a quantum cluster algebra.
\begin{conjecture}
In the fixed setting described above $K\otimes_{\mathcal{A}} K_0(\w\textrm{-proj})$, and hence $\btgh$, has the structure of a quantum cluster algebra.
\end{conjecture}
In \cite{KKKO}, \cite{KKKOII} the authors show that $\mathscr{C}_w$ is a monoidal categorification of $\mathcal{A}_v(\mathfrak{n}(w))$ as a quantum cluster algebra, where $\mathscr{C}_w$ is a certain subcategory of $\textbf{R}$-gmod labelled by $w \in W$ (here $W$ denotes the corresponding Weyl group). This means there exists a set of simple modules $\{M_k\}_{k \in J}$ in $\mathscr{C}_w$, labelled by a finite index set $J=J_{\textrm{fr}}\sqcup J_{\textrm{ex}}$ which has frozen and exchangeable parts, together with a $J\times J$ skew-symmetric matrix $L$ and a $J\times J_{\textrm{ex}}$ matrix $\widetilde{B}$, with a skew-symmetric principal part, which both have entries in $\mathbb{Z}$ and which satisfy certain conditions. This data is called a \textit{quantum monoidal seed} in $\mathscr{C}_w$ and will correspond to the initial seed of the quantum cluster algebra. Furthermore, we can mutate in direction $k\in J_{\textrm{ex}}$; we replace $M_k$ with a simple module $M_k^\prime$ which fits into two certain short exact sequences, using $R$-matrices, so that when we view the Grothendieck group we have a quantum cluster algebra type relation. We also replace $L$, $\widetilde{B}$ with matrices $\mu_k(L)$, $\mu_k(\widetilde{B})$ respectively. Moreover we are able to mutate repeatedly in every direction $k\in J_{\textrm{ex}}$ and there is an algebra isomorphism $\mathbb{Z}[v^{\pm \frac{1}{2}}]\otimes_{\mathbb{Z}[v^{\pm 1}]} K_0(\mathscr{C}_w) \cong \mathcal{A}_v(\mathfrak{n}(w))$. In particular the Grothendieck ring of $\mathscr{C}_w$ has a quantum cluster algebra structure. They show that the simple modules in the quantum monoidal seed are given by certain modules called \textit{determinantial modules}. They define the determinantial modules to be the simple modules corresponding to the quantum unipotent minors under the isomorphism $K_0(\mathscr{C}_w) \cong \mathcal{A}_v(\mathfrak{n}(w))_{\mathbb{Z}[v^{\pm 1}]}$. By the work of Gei{\ss}-Leclerc-Schr{\"o}er \cite{GeissLeclercSchroer} the quantum unipotent minors are the cluster variables in the initial seed of the algebra $\mathcal{A}_v(\mathfrak{n}(w))$ and lie in the upper global basis.
\\
\\
In this paper we have shown that the Grothendieck group of the category of finitely generated graded $\w$-modules has a $\mathbb{Z}[v,v^{-1}]$-algebra structure which is compatible with the algebra structure on the Grothendieck group of the category of finitely generated graded $\textbf{R}$-modules. In other words, we have the following algebra isomorphsims
\begin{equation*}
K_0(\w\textrm{-gmod})\cong K_0(\textbf{R}\textrm{-gmod})\cong \avn_{\mathbb{Z}[v^{\pm 1}]}.
\end{equation*}
Consequently, there exist certain subcategories of $\w\textrm{-gmod}$ which yield monoidal categorifications of the subalgebras $\avnw_{\mathbb{Z}[v^{\pm 1}]}$, as quantum cluster algebras. In particular, the Grothendieck group of $\mathscr{C}^B_w$ has a quantum cluster algebra structure, where $\mathscr{C}^B_w$ is the full subcategory of $\w\textrm{-gmod}$ consisting of modules induced from $\mathscr{C}_w$, and there is an algebra isomorphism $K_0(\mathscr{C}_w^B)\cong K_0(\mathscr{C}_w)$. Note that $\mathscr{C}^B_w$ is closed under the convolution product; for $M=\textrm{Ind}(P)$, $N=\textrm{Ind}(Q) \in \mathscr{C}^B_w$, where $P,Q \in \mathscr{C}_w$, we have $M\circ N=\textrm{Ind}(P)\circ \textrm{Ind}(Q)=\textrm{Ind}(P\circ Q)$, and indeed $P\circ Q \in \mathscr{C}_w$ since $\mathscr{C}_w$ is closed under the convolution product. We claim that a quantum monoidal seed in $K_0(\mathscr{C}_w^B)$ is precisely that induced from $K_0(\mathscr{C}_w)$.
\\
\\
We end this subsection with an example to demonstrate how the mutation process should take place in the category of VV algebra modules.
\begin{example}
Let $\beta=\lambda+\lambda^{-1}, \gamma=p^2\lambda+p^{-2}\lambda^{-1}, \delta=p^4\lambda+p^{-4}\lambda^{-1} \in \thetani$. Then let $L(0),L(2),L(4)$ be the two-dimensional simple $\w_\beta,\w_\gamma,\w_\delta$-modules, respectively, from Examples \ref{example:VV example} and \ref{example:VV example II}. So, for example, $L(0)$ is the two-dimensional $\vvb$-module spanned by $e(0),\pi e(0)$. Similarly for $L(2)$ and $L(4)$. One can check that $L(0)\circ L(2)$ and $L(2)\circ L(0)$ are both length two modules so that each has a simple head and simple socle. We write $L(0)\diamond L(2)$ to denote the head of $L(0)\circ L(2)$. Similarly, $L(2)\diamond L(0)$ denotes the head of $L(2)\circ L(0)$.
\\
\\
Put $M(0):=L(0)$, $M(2):=L(0)\diamond L(2)$, $M(4):=L(2)\diamond L(0)$.
\\
\\
$M(2)$ has basis $\{e(e(\lambda)\otimes e(p^2\lambda)),\pi(e(\lambda)\otimes e(p^2\lambda)),\sigma_1\pi(e(\lambda)\otimes e(p^2\lambda)),\pi\sigma_1\pi(e(\lambda)\otimes e(p^2\lambda))\}$.
\\
\\
$M(4)$ has basis $\{ e(e(p^2\lambda)\otimes e(\lambda)),\pi(e(p^2\lambda)\otimes e(\lambda)),\sigma_1\pi(e(p^2\lambda)\otimes e(\lambda)),\pi\sigma_1\pi(e(p^2\lambda)\otimes e(\lambda)) \}$.
\\
\\
We have a short exact sequence
\begin{equation*}
0\longrightarrow M(4)\longrightarrow M(0)\circ L(2) \xrightarrow{r_{M(0),L(2)}} M(2) \longrightarrow 0.
\end{equation*}
We have therefore mutated in direction $L(0)$ and the new module which have swapped $L(0)$ for is $M(0)^\prime=L(2)$.
\end{example}
\subsection{Convolution product in other settings.}
Throughout this paper we have imposed the restriction $p,q \not\in I$. We now want to know if it is possible to define a convolution product on modules over VV algebras in general, or in cases where we do not have these restrictions. We could, for example, again try to use the Morita equivalences established in \cite{Walker} and define a product using the convolution product for KLR algebras. However we soon encounter problems when using this method. For example, suppose we assume $q \in I$, $p \not\in I$ and we permit only those $\nu \in \thetani$ in which the coefficient of $q$ is 1. By \cite{Walker}, Theorem 2.43 we have Morita equivalence between $\vv$ and $\klrplus\otimes_{k[z]} A$. Here $A$ is the path algebra of the quiver consisting of two vertices and two arrows where each vertex is the source of one arrow and the target of the other. Defining the convolution product for VV algebra modules via this Morita equivalence would mean that the resulting product would be a module over a VV algebra associated to some $\nu$ in which the coefficient of $q$ is equal to 2. In other words, the product would not be closed when considering subcategories of $\w\textrm{-Mod}$ with the stated restrictions on the multiplicity of $q$. The same comments can be made about the other Morita equivalences established in \cite{Walker} and so it appears that some other method should be used in order to construct a product in these settings.
\subsection{Categorification of $\btg$.}
Recall that Varagnolo and Vasserot proved there is a $\btg$-module isomorphism $K\otimes_{\mathcal{A}} K_0(\w\textrm{-proj}) \cong \btg/\mathfrak{m}$, where $\mathfrak{m}$ is a maximal left submodule of $\btg$. There is a natural surjection of $\btg$-modules $\pi:\btg \twoheadrightarrow \btg/\mathfrak{m}$. Does there exist an algebra whose module category categorifies $\btg$ as a $\mathbb{Q}(v)$-algebra? 

\bibliography{C:/Users/Ruari/Dropbox/MyStyleFile/MyBibliography}

\providecommand{\bysame}{\leavevmode\hbox to3em{\hrulefill}\thinspace}
\providecommand{\MR}{\relax\ifhmode\unskip\space\fi MR }
\providecommand{\MRhref}[2]{%
  \href{http://www.ams.org/mathscinet-getitem?mr=#1}{#2}
}
\providecommand{\href}[2]{#2}
\begin{thebibliography}{KKKO15}

\bibitem[BZ05]{BerensteinZelevinsky}
A.~Berenstein and A.~Zelevinsky, \emph{Quantum cluster algebras}, Adv. Math.
  \textbf{195} (2005), no.~2, 405--455. \MR{2146350}

\bibitem[EK06]{EnomotoKashiwara}
N.~Enomoto and M.~Kashiwara, \emph{{Symmetric crystals and affine Hecke
  algebras of type B}}, Proc. Japan Acad. Ser. A Math. Sci. 82 (2006), no. 8,
  131–136. (2006).

\bibitem[EK09]{EnomotoKashiwara2}
\bysame, \emph{{Symmetric crystals for $\mathfrak{gl}_\infty$}}, Publ. Res.
  Inst. Math. Sci. 44, no. 3, 837–891 (2009).

\bibitem[FZ02]{FominZelevinskyI}
S.~Fomin and A.~Zelevinsky, \emph{Cluster algebras. {I}. {F}oundations}, J.
  Amer. Math. Soc. \textbf{15} (2002), no.~2, 497--529. \MR{1887642}

\bibitem[GLS13]{GeissLeclercSchroer}
C.~Gei{\ss}, B.~Leclerc, and J.~Schr{\"o}er, \emph{Cluster structures on
  quantum coordinate rings}, Selecta Math. (N.S.) \textbf{19} (2013), no.~2,
  337--397. \MR{3090232}

\bibitem[Her]{HernandezBourbaki}
D.~Hernandez, \emph{{A}dvances in {R}-matrices and their applications (after
  {M}aulik-{O}kounkov, {K}ang-{K}ashiwara-{K}im-{O}h,$\ldots$)}, preprint
  arXiv:1704.06039 [math.{QA}].

\bibitem[HL10]{HernandezLeclerc}
D.~Hernandez and B.~Leclerc, \emph{{C}luster algebras and quantum affine
  algebras}, Duke Math. J. \textbf{154} (2010), no.~2, 265--341. \MR{2682185}

\bibitem[KKK]{KKKI}
S.-J. Kang, M.~Kashiwara, and M.~Kim, \emph{{S}ymmetric {q}uiver {H}ecke
  algebras and {R}-matrices of quantum affine algebras}, preprint
  arXiv:1304.0323v2 [math.{RT}].

\bibitem[KKKOa]{KKKO}
S.-J. Kang, M.~Kashiwara, M.~Kim, and S.-J. Oh, \emph{Monoidal categorification
  of cluster algebras}, preprint arXiv:1412.8106 [math.{RT}].

\bibitem[KKKOb]{KKKOII}
\bysame, \emph{Monoidal categorification of cluster algebras {II}}, preprint
  arXiv:1502.06714 [math.RT].

\bibitem[KKKO15]{SimplicityOfHeadsSocles}
\bysame, \emph{{S}implicity of heads and socles of tensor products}, Compos.
  Math. \textbf{151} (2015), no.~2, 377--396. \MR{3314831}

\bibitem[KL09]{KhovanovLauda}
M.~Khovanov and A.~Lauda, \emph{A diagrammatic approach to categorification of
  quantum groups. {I}}, Represent. Theory \textbf{13} (2009), 309--347.
  \MR{2525917 (2010i:17023)}

\bibitem[Mak01]{ChiKinMak}
C.K. Mak, \emph{Quasi-parabolic subgroups of {$G(m,1,r)$}}, J. Algebra
  \textbf{246} (2001), no.~2, 471--490. \MR{1872111 (2002j:20072)}

\bibitem[MO]{OkounkovMaulik}
D.~Maulik and A.~Okounkov, \emph{{Q}uantum {G}roups and {Q}uantum
  {C}ohomology}, preprint arXiv:1211.1287v2 [math.{AG}].

\bibitem[Rou]{Rouquier2KacMoody}
R.~Rouquier, \emph{2-{K}ac-{M}oody algebras}, preprint arXiv:0812.5023v1
  [math.{RT}].

\bibitem[Rou12]{Rouquier}
\bysame, \emph{Quiver {H}ecke algebras and 2-{L}ie algebras}, Algebra Colloq.
  \textbf{19} (2012), no.~2, 359--410. \MR{2908731}

\bibitem[VV11a]{VaragnoloVasserot}
M.~Varagnolo and E.~Vasserot, \emph{{Canonical bases and affine Hecke algebras
  of type B}}, Invent. Math. 185, no. 3, 593–693 (2011).

\bibitem[VV11b]{CanonicalBasesAndKLR_Algebras}
\bysame, \emph{Canonical bases and {KLR}-algebras}, J. Reine Angew. Math.
  \textbf{659} (2011), 67--100. \MR{2837011}

\bibitem[Wal]{Walker}
R.~Walker, \emph{On {M}orita {E}quivalences {B}etween {KLR} {A}lgebras and {VV}
  {A}lgebras}, preprint arXiv:1603.00796 [math.{RT}].

\end{thebibliography}
\bibliographystyle{amsalpha}
\addcontentsline{toc}{section}{References}


\end{document}